\documentclass[a4paper]{article}

\usepackage[utf8]{inputenc}
\usepackage[english]{babel}
\usepackage{cite}
\usepackage[nottoc]{tocbibind}
\usepackage{amsmath}
\usepackage{amsfonts}
\usepackage{amssymb}
\usepackage{amsthm}
\usepackage{anysize}
\usepackage{centernot}
\usepackage{mathtools}
\usepackage{graphicx}
\usepackage{dcolumn}
\usepackage{enumitem}
\usepackage{tikz}
\usepackage{ytableau}
\usepackage{genyoungtabtikz}

\DeclareMathOperator{\Lie}{Lie}
\DeclareMathOperator{\pad}{p}
\DeclareMathOperator{\suc}{s}
\DeclareMathOperator{\supp}{Supp}
\DeclareMathOperator{\spa}{Span}

\DeclareMathOperator{\ch}{char}

\DeclareMathOperator{\id}{Id}

\DeclareMathOperator{\adj}{Ad}
\DeclareMathOperator{\N}{N}
\newcommand{\cdr}[1]{\overline{#1}}

\newtheorem{theorem}{Theorem}[section]
\newtheorem*{theorem*}{Theorem}
\newtheorem{Lemma}[theorem]{Lemma}
\newtheorem{Proposition}[theorem]{Proposition}
\newtheorem{Corollary}[theorem]{Corollary}
\theoremstyle{definition}
\newtheorem{Definition}[theorem]{Definition}
\theoremstyle{example}

\newcommand{\matr}[1]{\mathbf{#1}}

\newtheorem{manualtheoreminner}{Theorem}
\newenvironment{manualtheorem}[1]{%
  \renewcommand\themanualtheoreminner{#1}%
  \manualtheoreminner
}{\endmanualtheoreminner}

\title{The $B$-orbits on a Hermitian symmetric variety in characteristic $2$}
\author{Michele Carmassi}

\begin{document}
\maketitle
\begin{abstract}
Let $G$ be a reductive linear algebraic group over an algebraically closed field $\mathbb{K}$ of characteristic $2$.
Fix a parabolic subgroup $P$ such that the corresponding parabolic subgroup over $\mathbb{C}$ has abelian unipotent radical and fix a Levi subgroup $L\subseteq P$.
We parametrize the orbits of a Borel $B\subseteq P$ over the Hermitian symmetric variety $G/L$ supposing the root system $\Phi$ is irreducible.
For $\Phi$ simply laced we prove a combinatorial characterization of the Bruhat order over these orbits.
We also prove a formula to compute the dimension of the orbits from combinatorial characteristics of their representatives.
\end{abstract}
\section{Introduction}

Let $G$ be a connected reductive linear algebraic group over an algebraically closed field $\mathbb{K}$.
Denote with $\mathfrak{g}$ the Lie algebra of $G$ and fix a maximal torus $T$.
This defines a decomposition
\[\mathfrak{g}=\mathfrak{t}\oplus\bigoplus_{\alpha\in\Phi}\mathfrak{u}_\alpha\]
where $\mathfrak{t}$ is the Lie algebra of $T$,  the $\mathfrak{u}_\alpha$ are the root spaces and $\Phi=\Phi(G,T)$ is the root system of $G$.
Fix a Borel subgroup $B\supseteq T$ which is the same as a basis $\Delta$ of $\Phi$ or a set of positive roots $\Phi^+\subseteq \Phi$.
Recall that the Weyl group $W$ of $\Phi$ can be identified with $N_G(T)/T$ and that the simple reflections $s_\alpha$ with $\alpha\in \Delta$ generate $W$.

We know that subsets $S\subseteq \Delta$ correspond to parabolic $P_S\supseteq B$.
Consider a parabolic $P=P_S$ such that $S=\Delta\setminus\left\lbrace\alpha_P\right\rbrace$ where $\alpha_P$ has the property that for every $\alpha=\sum_{\delta\in\Delta}a_\delta\delta\in \Phi^+$ we have $a_{\alpha_P}\leq 1$.
This implies that the unipotent radical $P^u$ of $P$ as well as the Lie algebra $\mathfrak{p}^u$ of $P^u$ are abelian.
We denote with $\Psi$ the set $\left\lbrace \alpha\in \Phi^+\mid \mathfrak{u}_\alpha\subseteq \mathfrak{p}^u\right\rbrace$.
In this hypothesis consider a Levi subgroup $L\subseteq P$.
If $\ch\mathbb{K}\neq 2$ there is an involution $\Theta\colon G\longrightarrow G$ such that $L$ is the identity component of the set of fixed points $G^\Theta$ and for this reason the quotient $G/L$ is called a \textit{Hermitian symmetric variety}.
We extend this terminology to the case of characteristic $2$, which is the case we are interested in.


The Borel subgroup $B$ acts on $G/L$ by multiplication and on $\mathfrak{p}^u$ through the adjoint representation.
In both cases the orbits are finite and we can define an order by stating that $\mathcal{O}<\mathcal{O}'$ if and only if $\mathcal{O}\subseteq \cdr{\mathcal{O}'}$ where on the right we have the Zariski closure of $\mathcal{O}'$.
This is called the \textit{Bruhat order}.

The situation is quite similar to the well-known case of a flag variety $G/B$.
In this case we have the \textit{Bruhat decomposition}
\[G/B=\bigsqcup_{v\in W} BvB/B\]
and $BvB/B<BwB/B$ if and only if $v<w$ where the order in $W$ (which is still called Bruhat order) has the following combinatorial characterization:
for every reduced expression $w=s_{\alpha_1}\cdots s_{\alpha_n}$ there must be a subsequence $1\leq i_1<\ldots<i_r\leq n$ such that $v=s_{\alpha_{i_1}}\cdots s_{\alpha_{i_r}}$.
More generally, if $P\supseteq B$ is any parabolic subgroup define $W_P$ the Weyl group of $P$ and $W^P$ the set of minimal length representatives of the quotient $W/W_P$.
Then we have the decomposition
\[G/P=\bigsqcup_{v\in W^P} BvP/P\]
and $BvP/P<BwP/P$ if and only if $v<w$.

Return now to our parabolic $P$ with abelian unipotent radical and consider the projection $\pi\colon G/L\longrightarrow G/P$.
Its fibres are isomorphic to $\mathfrak{p}^u$ (see \cite{Seitz}) and, if we denote with $\omega^P$ the longest element in $W^P$, the stabilizer in $B$ of $\omega^PP\in G/P$ is exactly $B_L=B\cap L$.
It follows that the $B_L$-orbits in $\mathfrak{p}^u$ correspond exactly to the $B$-orbits in $B\omega^PP$.
But in our hypothesis $P^u$ acts trivially on $\mathfrak{p}^u$ so the $B_L$-orbits and the $B$-orbits on $\mathfrak{p}^u$ coincide.
Moreover, the correspondence between the orbits is order-preserving.
It follows that we can see the $B$-orbits on $\mathfrak{p}^u$ as a subset of the $B$-orbits in $G/L$.
Actually, we can define for every $v\in W^P$ subgroups $B_v=v^{-1}Bv\cap P$ and actions of each $B_v$ on $\mathfrak{p}^u$ (Equation \ref{action}, page \pageref{action}) such that every $B$-orbit in $G/L$ corresponds to a unique $B_v$-orbit in $\mathfrak{p}^u$ for some $v\in W^P$.
It follows that the problem of parametrizing the $B$-orbits in $G/L$ can be reduced to parametrizing the $B_v$-orbits in $\mathfrak{p}^u$.
Note that $B_{\omega^P}=B_L$ and the action of $B_L$ coincides with the action of $B$ so among these actions there is also the adjoint action of $B$ on $\mathfrak{p}^u$ we were interested in from the beginning.

Until now, the results we gave were not dependent on the characteristic of the base field $\mathbb{K}$.
Suppose now that the characteristic is different from $2$.
Then the $B$-orbits on $\mathfrak{p}^u$ were parametrized by Panyushev \cite[Theorem 2.2]{Pan} while the orbits on the Hermitian symmetric variety were parametrized by Richardson and Springer in \cite{RS2}.
We won't use directly the parametrization by Richardson and Springer, so we won't introduce it.
We will use instead a reformulation of it by Gandini and Maffei (see \cite[Definition 1.1]{GM}) which uses as parameters the \textit{admissible pairs}.
These are pairs $(v,S)$ with $v\in W^P$ and $S$ an orthogonal subset of roots for which $v(S)<0$.
Here for orthogonal subset we mean a subset such that $\alpha$ and $\beta$ are orthogonal for every $\alpha\neq\beta\in S$.
Gandini and Maffei also proved a combinatorial characterization of the Bruhat order in $G/L$ which was itself a conjecture by Richardson and Ryan (see \cite[Conjecture 5.6.2]{RS2}).

In this paper, we will study the $B$-orbits on $G/L$ and $\mathfrak{p}^u$ in the hypothesis of $\ch(\mathbb{K})=2$.
The main objective will be to give a parametrization of the orbits both on $\mathfrak{p}^u$ and on $G/L$.
The results are divided in relations to the type of the root system $\Phi$ which we will always suppose irreducible.

Fix an element $e_\alpha\neq 0$ for every root space $\mathfrak{u}_\alpha\subseteq \Lie(G)=\mathfrak{g}$ and for $S\subseteq \Phi$ define $e_S=\sum_{\alpha\in S} e_\alpha$.
When the root system is simply laced, that is of type $\bf{ADE}$, we will obtain the following parametrization for $\mathfrak{p}^u$:
\begin{manualtheorem}{\ref{parametrizzPu}}
There is a correspondence:
\begin{align*}
\left\lbrace S \text{ orthogonal }\mid S\subseteq \Psi\right\rbrace &\leftrightarrow \left\lbrace B\text{-orbits in }\mathfrak{p}^u\right\rbrace\\
S &\mapsto Be_S
\end{align*}
\end{manualtheorem}

Now, for $S\subseteq \Psi$ orthogonal define $x_S=\exp(e_S)L/L$.
By Theorem 5.3 in \cite{Seitz}, the exponential map $\exp\colon \mathfrak{p}^u\longrightarrow P^u$ is well defined even in characteristic $2$.
Then, for $G/L$ we will prove:
\begin{manualtheorem}{\ref{parametrizzazioneADE}}
There is a correspondence:
\begin{align*}
\left\lbrace (v,S)\mid v\in W^P, v(S)<0, S\text{ orthogonal}\right\rbrace &\leftrightarrow \left\lbrace B\text{-orbits in }G/L\right\rbrace\\
(v,S)&\mapsto Bvx_S
\end{align*}
\end{manualtheorem}

These coincide with the parametrizations in \cite{GM} and we will see that the proof is similar.

For the simply laced case we will also study the Bruhat order in $G/L$ and, by restriction, on $\mathfrak{p}^u$.
The fact that the parametrization doesn't depend on the characteristic makes it easy to conjecture that the characterization of the order remains the same.
That is what actually happens, but the proof is not as straightforward as one might think because some intermediate results used by Gandini and Maffei which come from \cite{RS2} don't have a clear analogue in this setting.
In the end, we will prove the characterization by showing that the Bruhat order in characteristic $2$ and the Bruhat order in characteristic different from $2$ define the same order on the set of admissible pairs which are parameters for both.

Following \cite{RS}, to every admissible pair we can associate an involution in $W$ as
\[\sigma_{v(S)}=\prod_{\gamma\in v(S)}s_\gamma\]
Moreover, given a $w\in W$ denote with $[w]^P$ the representative  in $W^P$ of the coset $wW_P$.
The following result coincides with \cite[Theorem 1.3]{GM}:

\begin{manualtheorem}{\ref{ordineADE}}
In the simply laced case we have
\[Bux_R\leq Bvx_S\Leftrightarrow \sigma_{u(R)}\leq\sigma_{v(S)}\text{ and }[v\sigma_S]^P\leq[u\sigma_R]^P\leq u\leq v\]
In particular, the Bruhat order in the simply laced case doesn't depend on the characteristic of the base field.
\end{manualtheorem}
Note that this also gives a characterization of the order in $\mathfrak{p}^u$ if we restrict to $u=v=\omega^P$.

The type $\bf{C}$ case is more complicated because the Panyushev parametrization in orthogonal subsets fails.
To parametrize these orbits we introduce another definition of admissible pairs.
Note that for every short root $\gamma\in \Phi^+$ there is a unique long root $\delta$ such that $\delta-\gamma\in\Phi^+$.
In this case, we put $\suc(\gamma)=\delta$.

\begin{Definition}[Definition \ref{admissibleC} and Definition \ref{fulladmissibleC}]\label{parametrizzazioneC}
Let $\Phi^+(v)=\left\lbrace \alpha \in \Psi \mid v(\alpha)<0\right\rbrace$ and $S\subseteq \Phi^+(v)$. Then $S$ is \textit{full admissible} (for $v$) if $S$ can be partitioned as $X(S)\sqcup Z(S)$ where:
\begin{enumerate}

\item $X(S)$ is orthogonal;
\item every element of $Z(S)$ is a long root $\beta$ and for every $\beta\in Z(S)$ exists a $\alpha$ in $X(S)$ and $\gamma\in\Phi^+_P$ verifying $\beta=\alpha+\gamma$. This element is unique, so define $\pad(\beta)=\alpha$;
\item for every $\gamma\in X(S)$ long and $\alpha\in S$ short such that $\suc(\alpha)\in\Phi^+(v)$, $\suc(\alpha)>\gamma$ and $\alpha\nless\gamma$ we have $\suc(\alpha)\in Z(S)$;
\item for every $\gamma\in Z(S)$ and $\alpha\in S$ short such that $\suc(\alpha)\in\Phi^+(v)$, $\suc(\alpha)>\gamma$ and $\pad(\gamma)\nless\alpha$ we have $\suc(\alpha)\in Z(S)$.
\end{enumerate}
\end{Definition}

The most important result of this paper are Theorem \ref{primaparte} and \ref{secondaparte}, which together imply that in type $\matr{C}$:

\begin{theorem}
There is a correspondence:
\begin{align*}
\left\lbrace (v,S) \mid S\text{ is full admissible for }v\right\rbrace &\leftrightarrow B\text{-orbits in }G/L\\
(v,S)&\mapsto Bvx_S
\end{align*}
\end{theorem}

The type $\bf{B}$ case is similar to the type $\bf{C}$ case, but the combinatorics are simpler.
We obtain a parametrization which is similar to the one above and that can be proved in a more manual way (Theorem \ref{parametrizzazioneB}, page \pageref{parametrizzazioneB}).

We also show a generalization of the dimensional formula \cite[Lemma 7.2]{RS} that is true in any characteristic (Theorem \ref{theodim}, page \pageref{theodim}).
From this, we prove two formulas to compute the dimension of the orbits in the type $\bf{B}$ and $\bf{C}$ cases that depend only on the combinatorial characteristics of the representatives that parametrize the orbit.
Again, the most interesting result is the type $\bf{C}$ one which is the following:

\begin{manualtheorem}{\ref{theodimC}}
Let $v\in W^P$ and $S$ a full admissible for $v$.
Then:
\[\dim(Bvx_S)=\#\Psi+L(\sigma_{v(X(S))})-\#S_s+\#Z(S)\]
\end{manualtheorem}
Here we denote with $S_s$ the set of short roots in $S$ and with $L(\sigma_{v(X(S))})=\frac{l(\sigma_{v(X(S))})+\#S}{2}$ the length of $\sigma_{v(X(S))}$ as an involution.

The paper is organized as follows.
After a brief introduction of notations in section $2$, we recall and expand some results from \cite{GM} and \cite{RS} that are independent from the characteristic (section 3).
These facts are mostly about the specific combinatorics of the roots and the Weyl groups and will be of great use later.
Here, we will also introduce the action of the minimal parabolic subgroups.
The idea of studying this action to understand the Bruhat order is from Richardson and Springer and is central both in \cite{GM} and in \cite{RS}.
Unsurprisingly, it will also be the core tool we will use in this paper.

In section 4 we will prove the results regarding the simply laced case, while in sections 5 and 6 we will describe the parametrization of the $B$-orbits in the type $\bf{B}$ and $\bf{C}$ respectively.
The seventh and last section will be devoted to proving the dimensional formula and its corollaries.
\section{Notations}

Fix once and for all an algebraically closed field $\mathbb{K}$ of characteristic $2$ and a connected, reductive, linear algebraic group $G$ over $\mathbb{K}$.
In $G$, fix a torus $T$ and a Borel subgroup $B$ containing $T$.
They define a root system $\Phi=\Phi(G,T)$ and a basis $\Delta$ of $\Phi$, which is the same as a subset of positive roots $\Phi^+$.
We will often write $\alpha>0$ and $\alpha<0$ to mean that $\alpha\in\Phi^+$ and $\alpha\in-\Phi^+$ respectively.
The root system will always be reduced and irreducible.

For every root $\alpha\in \Phi$ we have a one-dimensional root space $\mathfrak{u}_\alpha$ in the Lie algebra $\mathfrak{g}$ of $G$ and a one-parameter subgroup $U_\alpha$ in $G$.
Recall that formally a one-parameter subgroup is a morphism of linear algebraic groups $u_\alpha\colon\mathbb{K}\longrightarrow G$ and $U_\alpha$ is just the image of this map.
Therefore, we will often use $u_\alpha(t)$ to denote an element in $U_\alpha$.
Fix once and for all representatives $e_\alpha\in \mathfrak{u}_\alpha$ and if $S\subseteq \Phi$ denote $e_S=\sum_{\alpha\in S} e_\alpha$.
As a sort of converse of this construction if $x\in\mathfrak{g}$ we know that $x$ can be written as $x=\sum_{\alpha\in\Phi}a_\alpha\alpha$ and we will call the \textit{support} of $x$, denoted as $\supp(x)$, the set $\left\lbrace \alpha\in\Phi\mid a_\alpha\neq 0\right\rbrace$.
If $M\subseteq \mathfrak{g}$ we will denote with $\supp(M)$ the union of $\supp(x)$ for all $x\in M$.

To every root system we can associate its Weyl group $W$.
It is defined as the subgroup of isometries generated by the reflections $s_\alpha$ that fix the hyperplane orthogonal to $\alpha$ and send $\alpha$ to $-\alpha$ for every root $\alpha\in\Phi$.
In this case, the Weyl group can be realized as the quotient $\N_G(T)/T$ where $\N_G(T)$ is the normalizer of $G$ in $T$.
For this reason, we will often treat a $v\in W$ as an element of $G$ by identifying it with a representative.
Every time we will do this, which representative we choose will be irrelevant.
Given a $v\in W$ we will denote with $\Phi^+(v)$ the set $\left\lbrace \alpha\in\Phi^+\mid v(\alpha)<0\right\rbrace$.

There is a correspondence between subsets of $S\subseteq \Delta$ and parabolic subgroups $P_S$ containing $B$ and in this correspondence algebraic facts about $P$ correspond to combinatorial properties of $S$.
For $\gamma\in\Phi$ and $\alpha\in\Delta$, denote with $\left[\gamma,\alpha\right]$ the coefficient of $\alpha$ in the unique expression of $\gamma$ as a linear combination of elements of $\Delta$.
There is a root $\theta\in\Phi$ called the \textit{highest root} identified by $\left[\theta,\alpha\right]\geq\left[\gamma,\alpha\right]$ for every $\gamma\in\Phi$ and $\alpha\in\Delta$.
Fix a parabolic group $P_S$ with $S=\Delta\setminus \left\lbrace\alpha_P\right\rbrace$ and $\alpha_P\in\Delta$ such that $\left[\theta,\alpha_P\right]=1$.
We will denote it $P$ without subscripts.
Note that with this hypothesis the unipotent radical $P^u$ of $P$ and its Lie algebra $\mathfrak{p}^u$ are abelian.
Moreover, every parabolic with abelian unipotent radical is obtained this way if $\Phi$ is simply laced or if $\ch\mathbb{K}\neq 2$ (see \cite[Lemma 2.2]{RRS}).
Note that these are the only parabolic subgroups of $G$ such that, for every field $\mathbb{F}$, if $G_\mathbb{F}$ is a linear algebraic group over $\mathbb{F}$ with the same root datum as $G$ and $P_\mathbb{F}$ is the parabolic subgroup corresponding to $P$, then $P_\mathbb{F}^u$ is abelian.

The parabolic $P$ admits a Levi decomposition $P=L\ltimes P^u$.
Here $L$ is called the \textit{Levi subgroup} of $P$ and is reductive.
Its root system $\Phi_P$ can be seen as the subsystem of $\Phi$ generated by $S$.
It follows that $S$ is a basis for $\Phi_P$ and we will denote it $\Delta_P$.
It is easy to see that 
\[\begin{array}{lcr}\Phi_P=\left\lbrace \gamma\in\Phi\mid \left[\gamma,\alpha_P\right]=0\right\rbrace & & \Phi^+\setminus \Phi_P=\left\lbrace \gamma\in\Phi\mid \left[\gamma,\alpha_P\right]=1\right\rbrace
\end{array}\]
Denote $\Psi=\Phi^+\setminus \Phi_P$ and note that if $\gamma,\delta\in\Psi$ and $\gamma-\delta\in\Phi$, then $\gamma-\delta\in \Phi_P$.
Also, we have
\[\mathfrak{p}^u=\bigoplus_{\gamma\in\Psi}\mathfrak{u}_\gamma\]

The choice of a parabolic subgroup $P_S=P$ gives naturally two subsets of the Weyl group.
The first is the subgroup of isometries generated by $\left\lbrace s_\alpha \mid \alpha\in S\right\rbrace$, which is the Weyl group associated to the Levi subgroup $L$.
We will denote this with $W_P$.

The second is the set of isometries 
\[W^P=\left\lbrace w\in W\mid w(\alpha)>0 \text{ for every }\alpha\in S\right\rbrace\]
They are related by $W=W^PW_P$.
Note that $w\in W^P$ if and only if $\Phi^+(w)\subseteq \Psi$ and there is a maximal element in $W^P$ which we note with $\omega^P$ and that is identified by $\Phi^+(\omega^P)=\Psi$.

\section{General results}
As said, $\mathbb{K}$ will be a characteristic $2$ field.
Note that if $\Phi$ is of type $\bf{G}_2$ then $G$ admits no parabolic subgroup of the type we are looking for, so we can suppose without loss of generality that $\Phi$ is not of this type.
Then we know that if $\mathbb{K}$ has odd or zero characteristic and $\alpha,\beta\in\Phi$ we have
\[u_\alpha(t).e_\beta=e_\beta+ate_{\beta+\alpha}+bt^2e_{\beta+2\alpha}\]
where $a\neq 0$ if and only if $\beta+\alpha\in\Phi$ and $b\neq 0$ if and only if $\beta+2\alpha\in\Phi$.
This is not true in characteristic $2$ and that's the ultimate reason for which the characterization by Panyushev in \cite{Pan} doesn't hold in this case.
The next result follows from \cite[Theorem 25.2]{Humphreys}.
\begin{Lemma}\label{IET}

Suppose $\alpha$, $\beta\in\Phi$ with $\alpha\neq -\beta$ and $u_\alpha(t)$ the one-parameter subgroup relative to $\alpha$.
Then
$$u_\alpha(t).e_\beta=e_\beta+ate_{\beta+\alpha}+bt^2e_{\beta+2\alpha}$$
where $a$ and $b$ don't depend on $t$ and:
\begin{enumerate}
\item $a=0$ if and only if $\beta+\alpha\notin \Phi$ or $\beta+\alpha\in \Phi$ and $\beta-\alpha\in \Phi$;
\item $b=0$ if and only if $\beta+2\alpha\notin \Phi$.
\end{enumerate}
\end{Lemma}

In this section we will cite many results from \cite{GM} regarding root systems and Weyl groups.
They clearly don't depend on the characteristic of $\mathbb{K}$ and will be of great importance in the next sections.

\begin{Proposition}[Proposition 2.5, \cite{GM}]
Let $v\in W^P$ and let $\alpha\in \Delta$ such that $s_\alpha v<v$. 
Put $\beta=-v^{-1}(\alpha)$.
Then $\beta$ is maximal in $\Phi^+(v)$ and minimal in $\Psi\setminus \Phi^+(s_\alpha v)$.

Vice versa:
\begin{enumerate}
\item if $\beta$ is maximal in $\Phi^+(v)$ then $\alpha=-v(\beta)\in \Delta$ and $s_\alpha v<v$;
\item if $\beta$ is minimal in $\Psi\setminus \Phi^+(v)$ then $\alpha=v(\beta)\in \Delta$ and $s_\alpha v>v$.
\end{enumerate}
\end{Proposition}

We denote with $<$ the Bruhat order on $W$.
\begin{Lemma}\label{propord}
Let $u,v\in W$ and suppose $u<v$.
For every $\alpha\in \Delta$ we have:
\begin{enumerate}
\item if $s_\alpha u>u$ and $s_\alpha v>v$ then $s_\alpha u<s_\alpha v$;
\item if $s_\alpha u<u$ and $s_\alpha v<v$ then $s_\alpha u<s_\alpha v$;
\item if $s_\alpha u>u$ and $s_\alpha v<v$ then $u\leq s_\alpha v$ and $s_\alpha u\leq v$.
\end{enumerate}
\end{Lemma}

Following \cite{RS} we will associate to every orbit a particular involution in $W$.
Note that many results in \cite{RS} are based on the characterization of $L$ as the identity component of the fixed points of an involution $\Theta\colon G\longrightarrow G$.
Such an involution doesn't necessarily exists when the characteristic of $\mathbb{K}$ is $2$ and we will give an alternative proof when the proof from Richardson and Springer is not suitable.

Now, let $\mathcal{I}\subseteq W$ be the subset of all involutions.
We can define an action of the set of simple reflections $s_\alpha$, for $\alpha\in \Delta$ on $\mathcal{I}$ in the following way:
$$s_\alpha\circ\sigma =\left\lbrace \begin{array}{ll}
                  s_\alpha\sigma & \text{if } s_\alpha\sigma=\sigma s_\alpha\\
                  s_\alpha\sigma s_\alpha &\text{if } s_\alpha\sigma\neq\sigma s_\alpha\\
                  
                \end{array}
              \right.
$$
Note that $s_\alpha\circ\sigma=\tau$ if and only if $s_\alpha\circ\tau=\sigma$.
\begin{Lemma}[Lemma 3.1, \cite{GM}]
Let $\alpha\in\Delta$ and $\sigma\in \mathcal{I}$.
Then $ s_\alpha \circ \sigma$ and $\sigma$ are always comparable.
Moreover, $s_\alpha \circ \sigma>\sigma$ if and only if $s_\alpha \sigma>\sigma$.
\end{Lemma}
Note that if $s_\alpha\sigma\neq \sigma s_\alpha$ then $s_\alpha\sigma s_\alpha>s_\alpha\sigma >\sigma$ and $s_\alpha\sigma s_\alpha>\sigma s_\alpha >\sigma$.

The action on involutions interacts with the Bruhat orders with properties similar to the one in Lemma \ref{propord}.

\begin{Lemma}[Lemma 3.2, \cite{GM}]\label{Bruhatcirc}
Let $\sigma,\tau\in \mathcal{I}$ and suppose $\sigma<\tau$.
For every $\alpha\in \Delta$ we have:
\begin{enumerate}
\item if $s_\alpha\circ\sigma>\sigma$ and $s_\alpha\circ\tau>\tau$ then $s_\alpha\circ\sigma<s_\alpha\circ\tau$;
\item if $s_\alpha\circ\sigma<\sigma$ and $s_\alpha\circ\tau<\tau$ then $s_\alpha\circ\sigma<s_\alpha\circ\tau$;
\item if $s_\alpha\circ\sigma>\sigma$ and $s_\alpha\circ\tau<\tau$ then $s_\alpha\circ\sigma\leq\tau$ and $\sigma\leq s_\alpha\circ \tau$.
\end{enumerate}
\end{Lemma}
We define the \textit{length} of an involution $\sigma$ as
$$L(\sigma)=\frac{l(\sigma)+\lambda(\sigma)}{2}$$
where $l(\sigma)$ is the usual length in $W$ and $\lambda(\sigma)$ is the dimension of the $(-1)$-eigenspace of $\sigma$ on $\Phi\otimes \mathbb{R}$.

\begin{Lemma}\label{invlength}
Let $\alpha\in\Delta$ and $\sigma\in \mathcal{I}$.

$$L(s_\alpha\circ\sigma)=\left\lbrace \begin{array}{ll}
										L(\sigma)+1 & \text{if }s_\alpha\circ\sigma>\sigma\\
										L(\sigma)-1 & \text{if }s_\alpha\circ\sigma<\sigma\\
									\end{array}
							\right.
$$
\end{Lemma}

To every set $S\subseteq\Psi$ of mutually orthogonal roots we can naturally attach the involution

$$\sigma_S=\prod_{\alpha\in S} s_\alpha$$
Note that if $\alpha$ and $\beta$ are orthogonal then $s_\alpha s_\beta=s_\beta s_\alpha$, so $\sigma_S$ is well defined.
The $(-1)$-eigenspace of such involution is generated by $S$ so we have
$$L(\sigma_S)=\frac{l(\sigma_S)+\#S}{2}$$
\begin{Lemma}[Lemma 3.6, \cite{GM}]\label{sommaradici}
Let $\beta, \beta'\in \Psi$ be orthogonal.
Then:
\begin{enumerate}
\item $\beta$ and $\beta'$ are strongly orthogonal, that is $\beta\pm\beta'\notin\Psi$;
\item if $\beta+\alpha\in \Phi$ for some $\alpha\in \Phi^+$ then $\beta'+\alpha\notin \Phi$;
\item if $\beta-\alpha\in \Psi$ for some $\alpha\in \Phi^+$ then $\beta'-\alpha\notin \Psi$.
\end{enumerate}
\end{Lemma}

We say that a subset $S\subseteq \Phi$ is \textit{strongly orthogonal} if all the roots in $S$ are strongly orthogonal.
In our analysis of the characteristic $2$ case, we will use the following result.
\begin{Corollary}[Corollary 3.9, \cite{GM}]\label{SugualeT}
Let $S,T\subseteq \Phi$ be strongly orthogonal and suppose $\sigma_S=\sigma_T$.
Then $S=T$.
\end{Corollary}

Now, consider the projection map $\pi\colon G/L\longrightarrow G/P$.
It is $B$-equivariant.
Recall that $G/P=\bigcup_{v\in W^P} BvP/P$ and for $v\in W^P$ define $B^v=vPv^{-1}\cap B$ the stabilizer of $vP\in G/P$ in $B$.
Then
$$\pi^{-1}(BvP/P)= BvP/L\cong B\times^{B^v} \pi^{-1}(vP)=B\times^{B^v} vP/L$$
Hence we have a bijection between the $B$-orbits in $BvP/L$ and the $B^v$ orbits in $vP/L$ which is compatible with the Bruhat order.
If we define $B_v=P\cap v^{-1}Bv$ then these orbits are in bijection with the $B_v$-orbits in $P/L$.
\begin{Lemma}[Lemma 4.1, \cite{GM}]
Let $v\in W^P$. 
Then $B_L=B_v\cap L$ and $B_v=B_L\ltimes U_v$ where $U_v$ is the subgroup of $P^u$ generated by the $U_\alpha$ with $\alpha\in \Psi\setminus\Phi^+(v)$.
\end{Lemma}
Note that the Lie algebra of $U_v$ is $\mathfrak{u}_v=\bigoplus_{\alpha\in \Psi\setminus\Phi^+(v)} \mathfrak{u}_\alpha$ and that if $\omega^P$ is the longest element in $W^P$, then the $B$ action is equal to the $B_L=B\cap L$-action which is by definition the $B_{\omega^P}$-action.

Let $\exp\colon\mathfrak{p}_u\longrightarrow P^u$ be the exponential map and compose it with the projection $\pi\colon G\longrightarrow G/L$.
We obtain an isomorphism $r_P\colon \mathfrak{p}_u\longrightarrow P/L$ that is not $P$-equivariant if we consider the adjoint action on $\mathfrak{p}_u$ and the left multiplication on $P/L$.
We want to define an action of $P$ on $\mathfrak{p}_u$ that makes $r_P$ a $P$-equivariant map.
Consider the isomorphisms
$$L\ltimes \mathfrak{p}_u\cong L\ltimes P^u\cong P$$
from left to right $(g,y)\longmapsto g\exp(y)$.
Note that with this identification we have $B_v=B_L\ltimes \mathfrak{u}_v$.
Let $(g,y)\in P$ and $x\in\mathfrak{p}_u$. 
Define the action

\begin{equation}\label{action}
(g,y).x=\adj_g(x+y)
\end{equation}

From this definition it is easy to see that if $u<v$ and $x\in\mathfrak{p}^u$ we have the containment $B_vx\subseteq B_ux$.
More precisely, suppose $u=s_\alpha v<v$ for some $\alpha\in \Delta$ and denote $\beta=v^{-1}(\alpha)$.
Then $B_u=B_v\ltimes \mathfrak{u}_\beta$ and 
\[B_ux=B_v.\mathfrak{u}_\beta.x=\bigcup_{t\in\mathbb{K}}B_v(x+te_\beta)\]

\begin{Lemma}[Lemma 4.2, \cite{GM}]\label{orderiso}
Let $v\in W^P$.
Then the map $B_ve\longmapsto Bv\exp(e)L/L$ is an order isomorphism between the $B_v$-orbits in $\mathfrak{p}_u$ and the $B$-orbits in $BvP/L$.
\end{Lemma}

We have the following formula regarding the dimensions:

\begin{Lemma}[Lemma 4.2, \cite{GM}]\label{dimensionrel}
Let $v\in W^P$ and $e$ an element in $\mathfrak{p}_u$.
Then the following formula holds
\[\dim Bv\exp(e)L/L=l(v)+\dim B_ve\]
\end{Lemma}

\begin{theorem}\label{BxB}
Let $x,y\in\mathfrak{p}_u$ with $B_vx\subseteq\cdr{B_vy}$.
Then 
\[Bv\exp(x)v^{-1}B\subseteq \cdr{Bv\exp(y)v^{-1}B}\]
\end{theorem}
\begin{proof}
Consider $x\in\cdr{B_vy}$ and apply the exponential.
We get $\exp(x)\in \cdr{B_v.\exp(y)}$ where $B_L\subseteq B_v$ acts by inner automorphisms and $U_v=\prod_{\alpha\in\left(\Psi\setminus \Phi^+(v)\right)}U_\alpha$ acts by multiplication.
We then have
$$
B_v.\exp(y)\subseteq B_L\exp(y)U_vB_L=B_L\exp(y)B_v\subseteq v^{-1}Bv\exp(y)v^{-1}Bv
$$
where we used the fact that $B_v\subseteq v^{-1}Bv$.
Then
$$v\exp(x)v^{-1}\in \cdr{Bv\exp(y)v^{-1}B}$$ 
and 
$$Bv\exp(x)v^{-1}B\subseteq \cdr{Bv\exp(y)v^{-1}B}\qedhere$$ 
\end{proof}

Now fix $v\in W^P$ and $S\subseteq \Phi^+(v)$ orthogonal. 
Define $g_{v(S)}=v\exp(e_S)v^{-1}$ and consider the double coset $Bg_{v(S)}B$.
The roots in $v(S)$ are negative, so $g_{-v(S)}\in B$.
$$Bg_{v(S)}B=Bg_{-v(S)}g_{v(S)}g_{-v(S)}B=B\sigma_{v(S)}B$$

where the last equality holds because $v(S)$ is orthogonal and the root vectors $e_\alpha$ verify 

$$\exp(e_{-\alpha})\exp(e_\alpha)\exp(e_{-\alpha})T/T=s_\alpha\in W$$

\begin{theorem}\label{orthogonalsubset}
Fix $v\in W^P$ and $S,T\subseteq \Phi^+(v)$ orthogonal subsets.
Then $B_ve_S=B_ve_T$ implies $S=T$.
\end{theorem}
\begin{proof}
By Theorem \ref{BxB} we get
$$Bvg_Sv^{-1}B=Bvg_Tv^{-1}B$$
and by the discussion above this implies
$$B\sigma_{v(S)}B=B\sigma_{v(T)}B$$
using the characterization of the order in the flag variety we get $v(S)=v(T)$, hence $S=T$ by Lemma \ref{SugualeT}.
\end{proof}

Given a simple root $\alpha\in \Delta$ we can define a parabolic subgroup $P_\alpha$ which is the subgroup generated by $B$ and $U_{-\alpha}$.
It is minimal among the parabolic subgroups that strictly contain $B$ and every such subgroup is obtained this way.

Now fix a $B$-orbit $BxL/L$ in $G/L$ and a simple root $\alpha\in\Delta$.
The minimal parabolic subgroup $P_\alpha$ acts on $G/L$, so the $B$-orbit $BxL/L$ is contained in the $P_\alpha$-orbit $P_\alpha xL/L$.
\begin{Proposition}[4.2,\cite{RS}]
The Borel subgroup $B$ acts on $P_\alpha xL/L$ with finitely many orbits, in fact there are at most $3$ $B$-orbits in $P_\alpha xL/L$.
\end{Proposition}

There must be a unique $B$-orbit $\mathcal{O}$ in $P_\alpha vx_S$ such that $\cdr{\mathcal{O}}=P_\alpha vx_S$.
We will call $\mathcal{O}$ the \textit{open orbit} of $P_\alpha vx_S$.

The dimension of $P_\alpha$ is $\dim B+1$, so $\dim Bvx_S\leq \dim P_\alpha vx_S\leq \dim Bvx_S+1$.
This implies that if $\mathcal{O}$ and $\mathcal{O}'$ are distinct $B$-orbits in $P_\alpha vx_S$ then they are comparable if and only if one of them is the open orbit.

\section{The simply laced case}

Suppose from now on that $G$ is a connected, reductive, linear algebraic group and that the root system $\Phi$ of $G$ is of type \textbf{ADE}.
Then if $\alpha,\beta\in \Phi$ and $\langle\alpha,\beta\rangle=0$ we know that $\alpha+\beta,\alpha-\beta\notin \Phi$ because $(\alpha+\beta,\alpha+\beta)=(\alpha,\alpha)+(\beta,\beta)$ and all the roots must have the same length.
It follows from Lemma \ref{IET} that

$$u_\alpha(t).e_\beta=e_\beta+ate_{\alpha+\beta}$$

where $a$ is a constant that depends only on $\alpha$ and $\beta$ and is non-null if and only if $\alpha+\beta$ is a root in $\Phi$.

For every $v\in W^P$ we can consider the action of $B_v$ on $\mathfrak{p}^u$ defined by Equation \ref{action}.

The following theorem and its proof coincide with Proposition 4.7 in \cite{GM} which itself follows directly from Lemma 4.6.
We opted to write the proof instead of just giving a reference because the hypothesis are different: in \cite{GM} it is $\ch(\mathbb{K})\neq 2$ and the result is valid for every root system.
Here, we have $\ch(\mathbb{K})=2$, but we are limited to simply laced root systems.

\begin{theorem}\label{paramBv}
For every $v\in W^P$ there is a correspondence

\begin{align*}
\left\lbrace S\subseteq \Phi^+(v) \text{ orthogonal}\right\rbrace &\leftrightarrow \left\lbrace B\text{-orbits in }\mathfrak{p}^u\right\rbrace\\
S &\mapsto B_ve_S
\end{align*}
\end{theorem}

\begin{proof}
We have already seen in Theorem \ref{orthogonalsubset} that such a map is injective.
We will show by induction on $l=l(v)$ that every $B_v$-orbit admits an element of the form $e_S$ where $S$ is orthogonal.

Suppose $l=0$, then $v=\id$ and the action of $B_{\id}$ is transitive on $\mathfrak{p}^u$. 

Now suppose $l>0$ and consider an orbit $\mathcal{O}$ in $\mathfrak{p}^u$.
Fix $\alpha\in\Delta$ for which $u=s_\alpha v<v$ and $\beta=v^{-1}(-\alpha)$.
By induction there is an orthogonal set $S$ such that $B_u\mathcal{O}=B_ue_S$.

If $S'=S\cup\left\lbrace \beta\right\rbrace$ is orthogonal then
$$B_ue_S=B_v(e_S+te_\beta)=B_ve_S\cup B_ve_{S'}$$
so it must be either $\mathcal{O}=B_ve_S$ or $\mathcal{O}=B_ve_{S'}$.

If $S'=S\cup\left\lbrace \beta\right\rbrace$ is not orthogonal then there must be $\gamma\in S$ such that $\langle\gamma,\beta\rangle\neq 0$.
But roots in $\Psi$ can't be added and $\beta$ is maximal so $\delta=\beta-\gamma$ must be a positive root in $\Phi_P$ and the one parameter subgroup $U_\delta$ is contained in $B_L\subseteq B_v$.
If we let such subgroup act on $e_S+te_\beta$ we get
\[u_\delta(s).(e_S+te_\beta)=u_\delta(s).(e_{S\setminus \left\lbrace\gamma\right\rbrace}+e_\gamma+te_\beta)=e_{S\setminus \left\lbrace\gamma\right\rbrace}+e_\gamma+ase_\beta+te_\beta=e_S+(as+t)e_\beta\]
where we used that $\beta=\delta+\gamma$ and that, by Lemma \ref{sommaradici}, $\tau+\delta$ is not a root for every $\tau\in S\setminus\left\lbrace\gamma\right\rbrace$, hence $u_\delta(s)$ fixes all $e_\tau$, $\tau\in S\setminus\left\lbrace\gamma\right\rbrace$.
Note that $a\neq 0$, so there is $s\in\mathbb{K}$ such that $u_\delta(s).(e_S+te_\beta)=e_S$.
Hence, $B_v(e_S+te_\beta)=B_ve_S$.

In both cases, the claim follows.
\end{proof}

As we said, Theorem \ref{paramBv} is true for simply laced root systems in any characteristic and for all root systems if the characteristic is not $2$.
If the characteristic is indeed $2$ and the root system is not simply laced, not only the proof fails, but the claim is false.
We will see this in Sections \ref{B} and \ref{tipoC} were we will see the non-simply laced root systems.

Theorem \ref{paramBv} easily implies the following parametrization.

\begin{theorem}\label{parametrizzPu}
There is a correspondence:
\begin{align*}
\left\lbrace S \text{ orthogonal }\mid S\subseteq \Psi\right\rbrace &\leftrightarrow B\text{-orbits in }\mathfrak{p}^u\\
S &\mapsto Be_S
\end{align*}
\end{theorem}
\begin{proof}
We know that the $B$-orbits and the $B_L$-orbits coincide and $B_L=B_{\omega^P}$ where $\omega^P$ is the longest element in $W^P$.
By definition $\Phi^+(\omega^P)=\Psi$ and the claim follows.
\end{proof}

Remember that a parametrization of the $B_v$-orbits also gives a parametrization of the $B$-orbits in the Hermitian symmetric variety $G/L$:

\begin{theorem}\label{parametrizzazioneADE}
There is a correspondence
\begin{align*}
\left\lbrace (v,S)\mid v\in W^P, S\subseteq \Phi^+(v), S\text{ orthogonal}\right\rbrace &\leftrightarrow B\text{-orbits in }G/L\\
(v,S)&\mapsto Bvx_S
\end{align*}
\end{theorem}

In the setting of simply laced root systems we will say that a pair $(v,S)$ with $v\in W^P$ and $S\subseteq \Phi^+(v)$ is \textit{admissible} if $S$ is orthogonal.
We will denote the set of admissible pairs with $V_L$.
From the theorem above and \cite[Proposition 4.7]{GM} the admissible pairs parametrize the $B$-orbits in $G/L$ regardless of the characteristic of the base field $\mathbb{K}$.

It is now natural to ask if the Bruhat order on the $B$-orbits depends on the characteristic.
The answer is no, but instead of proving the characterization directly, we will show that both orders agree as orders on $V_L$.
In the last part of this chapter, most proofs will mirror the equivalent proofs in \cite{GM}.
The most important original result is Lemma \ref{Springer} which in \cite{GM} derives from the existence of an involution that fixes $L$ which we do not have in characteristic $2$.

To start, fix a simple root $\alpha\in \Delta$.
Consider the minimal parabolic subgroup $P_S$ where $S=\left\lbrace\alpha\right\rbrace$ which we will denote for simplicity with $P_\alpha$.
Recall that we can let $P_\alpha$ act  on $Bvx_R$ on the left, obtaining $P_\alpha vx_R$ which is the union of, at most, three $B$-orbits, one of which is open and dense in $P_\alpha vx_R$.
The following definition comes from \cite{RS}, but, as before, we use the notation of \cite{GM}.

\begin{Definition}
Consider $P_\alpha vx_S\supseteq Bvx_S$ and define

\[m_\alpha(v,S)\doteqdot(u,T) \text{ if and only if }Bux_T \text { is open in }P_\alpha vx_S\]
\[\mathcal{E}_\alpha(v,S)\doteqdot\left\lbrace (u,T)\neq (v,S) \text{ admissible}\mid m_\alpha(u,T)=(v,S)\right\rbrace\]

\end{Definition}

Notice that from the correspondence between subsets of $\Delta$ and parabolic subgroups containing $B$ we know that $P_\alpha=B\cup Bs_\alpha B$ which can also be written as $P_\alpha=Bs_\alpha\cup BU_{-\alpha}$ where $U_{-\alpha}$ is the one parameter subgroup associated to $-\alpha$.

\begin{theorem}\label{frecciafacile}
Let $v\in W^P$ and $S,T\subseteq \Phi^+(v)$ be orthogonal subsets.
If $\cdr{B_ve_S}\supseteq B_ve_T$ then $\sigma_{v(S)}\geq \sigma_{v(T)}$.
\end{theorem}
\begin{proof}
By Theorem \ref{BxB} we have $Bvg_Tv^{-1}B\subseteq \cdr{Bvg_Sv^{-1}B}$.
Hence, $Bg_{v(T)}B\subseteq \cdr{Bg_{v(S)}B}$ and that implies $\sigma_{v(T)}\leq\sigma_{v(S)}$.
\end{proof}

Notice that with the correspondence $B_ve_S\longleftrightarrow Bvx_S$ which preserves the Bruhat order we also get that $Bvx_S\subseteq \cdr{Bvx_T}$ implies $\sigma_{v(S)}\leq\sigma_{v(T)}$.

We will use for simplicity a property that is true only in the simply laced case.
Suppose that $\alpha,\beta\in\Phi$ are not orthogonal.
Then it must be 
$$\langle\alpha,\beta\rangle=\langle\beta,\alpha\rangle=\pm 1$$
So we have
$$s_\alpha(\beta)=\beta-\langle \beta,\alpha\rangle \alpha=\left\lbrace \begin{array}{ll}
                  \beta-\alpha & \text{if } (\beta,\alpha)>0\\
                  \beta+\alpha &\text{if } (\beta,\alpha)<0\\
                  
                \end{array}
              \right.
$$

We will need a technical lemma.

\begin{Lemma}\label{Springer}
Let $(v,S)$ be an admissible pair and $\alpha\in\Delta$ a simple root.
If $s_\alpha\sigma_{v(S)}<\sigma_{v(S)}$ then $\mathcal{E}_\alpha(v,S)\neq \varnothing$.
\end{Lemma}
\begin{proof}
Recall that $s_\alpha\sigma_{v(S)}<\sigma_{v(S)}$ if and only if $v\sigma_S v^{-1}(\alpha)=\sigma_{v(S)}(\alpha)<0$.
Denote $\beta=v^{-1}(\alpha)$.
There are three cases:
\begin{description}
\item [Case $\beta\in \Psi$:]

Note that $\beta\notin \Phi^+(v)$ because $v(\beta)=\alpha>0$.
Consider

$$X=\left\lbrace \alpha\in S\mid s_\alpha(\beta)\neq \beta\right\rbrace=\left\lbrace \alpha_1,\ldots,\alpha_n\right\rbrace$$
Then $v\sigma_S (\beta)=v\sigma_X (\beta)=v(\beta -\alpha_1-\cdots -\alpha_n)$ because elements of $\Psi$ can't be summed.
But now $v(\beta -\alpha_1-\cdots -\alpha_n)=\alpha-v(\alpha_1)-\cdots-v(\alpha_n)$ and $v(\alpha_i)<0$ for every $i$, so $v(\beta -\alpha_1-\cdots -\alpha_n)>0$.
This contradicts the hypothesis;

\item [Case $\beta\in -\Psi$:]

Note that $-\beta\in \Phi^+(v)$.
If we define $X$ as above we have
$$v\sigma_S (\beta)=v\sigma_X (\beta)=\left\lbrace \begin{array}{ll}
                  v(\beta+\alpha_1+\cdots +\alpha_n)=\alpha+v(\alpha_1)+\cdots +v(\alpha_n) & \text{if } -\beta\notin S\\
                  v(-\beta) &\text{if } -\beta\in S\\
                  
                \end{array}
              \right.
$$

It follows that $v\sigma_S (\beta)<0$ if and only if $X\neq\varnothing$.

Consider $P_\alpha=Bs_\alpha\sqcup BU_{-\alpha}$.
Then
\begin{align*}
P_\alpha vx_S&=Bs_\alpha vx_S\sqcup BU_{-\alpha}vx_S\\
&=Bs_\alpha vx_S\sqcup BvU_{-\beta}x_S\\
&=Bs_\alpha vx_S\sqcup Bvx_S\cup \bigcup_{t\in\mathbb{K}^*}Bvu_{-\beta}(t)x_S
\end{align*}
We have $s_\alpha v<v$ and $s_\alpha v\in W^P$ so $Bs_\alpha vx_S$ can't be the open orbit.

Note that $u_{-\beta}(t)=\exp(te_{-\beta})$, so $u_{-\beta}(t)x_S=\exp(e_S+te_{-\beta})$.

Now if $-\beta\in S$, then 
$$B_v\left(e_S+te_{-\beta}\right)=\left\lbrace\begin{array}{ll}
B_ve_S &\text{ if }t\neq -1\\
B_ve_{S\setminus\left\lbrace -\beta\right\rbrace}&\text{ if }t= -1
\end{array}
\right.
$$
So if we set $S'=S\setminus\left\lbrace -\beta\right\rbrace$
$$P_\alpha vx_S=Bs_\alpha vx_{S'}\sqcup Bvx_{S'}\sqcup Bvx_S$$
Then the open $B$-orbit in $P_\alpha vx_S$ is $Bvx_S$ and $\left\lbrace(s_\alpha v,S'),(v,S')\right\rbrace=\mathcal{E}_\alpha(v,S)$.

Suppose $-\beta\notin S$.
We know that there is $\alpha\in S$ such that $(\alpha,-\beta)>0$.
But $-\beta$ is maximal in $\Phi^+(v)$ so there exist $\gamma \in \Delta_P$ such that $-\beta=\alpha+\gamma$.
So for every $t\in\mathbb{K}$ there is an $s\in\mathbb{K}$ such that 
\[u_{\gamma}(s).e_S=\exp(e_S+te_{-\beta})\]
Hence 
$$P_\alpha vx_S =Bs_\alpha vx_S \sqcup Bvx_{S\cup \left\lbrace -\beta\right\rbrace}$$
and $B_ve_{S\cup \left\lbrace -\beta\right\rbrace}=B_ve_S$ is the open orbit.
We obtain $\left\lbrace(s_\alpha v,S)\right\rbrace=\mathcal{E}_\alpha(v,S)$.


\item [Case $\beta\in \Delta_P$:]
As a first thing, note that this is the only remaining case.
In fact if $\beta\in\Phi_P$, then $\beta$ must be positive and if $\beta=\gamma_1+\cdots+\gamma_n$ is the decomposition in simple roots, then $v(\beta)=\alpha=v(\gamma_1)+\cdots+v(\gamma_n)$ and this is absurd because $v(\gamma_i)>0$ for every $i\in \left\lbrace 1,\ldots,n\right\rbrace$.

We have $s_\alpha v=vs_\beta$ so $[s_\alpha v]^P=v$ where $[s_\alpha v]$ is the representative in $W^P$ of the coset $s_\alpha vW_P$ in $W$.

There are four subcases depending on if $\beta$ can be added or subtracted to roots in $S$.
Remember that $\beta$ can be added or subtracted at most to a single root and it can't be added and subtracted to the same root.
\begin{description}
\item [$\beta$ can't be added nor subtracted to any root in $S$]\hfill

In this case $\beta$ is orthogonal to $S$. 
So $v\sigma_S(\beta)=v(\beta)>0$.
Hence, this case is impossible;
\item [$\beta$ can be added but not subtracted to a root in $S$]\hfill

Denote with $\gamma$ the root such that $\gamma+\beta\in \Phi$.
We have $v\sigma_S(\beta)=v(\gamma+\beta)$ and $\gamma+\beta\in \Phi^+(v)$ follows.
Write $P_\alpha=Bs_\alpha\sqcup BU_{-\alpha}$ 
\begin{align*}
P_\alpha vx_S&=Bs_\alpha vx_S\sqcup BvU_{-\beta}x_S\\
&=Bvx_{s_\beta(S)}\sqcup Bvx_S\sqcup \bigcup_{t\in\mathbb{K}^*}Bv\exp(u_{-\beta}(t).e_S)
\end{align*}
where we used the fact that $s_\beta\in L$ and $U_{-\beta}\subseteq L$.

We have $Bv\exp(u_{-\beta}(t).e_S)=Bvx_S$ for every $t$ and $s_\beta(S)=S'\cup\left\lbrace \gamma+\beta\right\rbrace$ where $S'=S\setminus \left\lbrace \gamma\right\rbrace$, so $s_\beta(S)\neq S$.

If we compute the involutions we get

$$\sigma_{vs_\beta(S)}=\sigma_{s_\alpha v(S)}=s_\alpha \sigma_{v(S)}s_\alpha=s_\alpha\circ\sigma_{v(S)}$$

The orbits $Bvx_S$ and $Bvx_{s_\beta(S)}$ must be comparable.
Using Theorem \ref{frecciafacile} we find that $\cdr{Bvx_S}\supseteq Bvx_{s_\beta(S)}$ so $\left\lbrace(v,s_\beta(S))\right\rbrace= \mathcal{E}_\alpha(v,S)$;

\item [$\beta$ can be subtracted but not added to a root in $S$]\hfill

Denote with $\gamma$ the root such that $\gamma-\beta\in \Phi$.
As before we have $v\sigma_S(\beta)=v(\beta-\gamma)$, but then both $v(\beta)=\alpha$ and $v(-\gamma)$ are positive, so this case is impossible.

\item [$\beta$ can be subtracted and added to two roots in $S$]\hfill

Denote with $\gamma_+$ the root such that $\gamma_++\beta\in\Phi$ and with $\gamma_-$ the root such that $\gamma_--\beta\in\Phi$.
We have $v\sigma_S(\beta)=v(\beta+\gamma_+-\gamma_-)<0$ so the root $\beta+\gamma_+-\gamma_-$ which is in $\Phi_P$ must be negative.
The orbit $P_\alpha vx_S$ can be again decomposed as
$$P_\alpha vx_S=Bvx_{s_\beta(S)}\sqcup Bvx_S\sqcup \bigcup_{t\in\mathbb{K^*}}Bv\exp(u_{-\beta}(t).e_S)$$

Consider $B_vu_{-\beta}(t).e_S$.
Given that $\delta=\gamma_--\beta-\gamma_+$ is positive we have 
$$u_\delta(t)e_{\gamma_+}=e_{\gamma_+}+ate_{(\gamma_--\beta)}$$
where $a\neq 0$ is a constant.
Note that $u_{-\beta}(t).e_S=e_S+ate_{(\gamma_--\beta)}$ as well.
It follows that for every $t_0\in\mathbb{K}^*$ there is a $t$ such that
$$u_\delta(t)e_S=u_{-\beta}(t_0).e_S$$
Again, in $P_\alpha vx_S$ we have two orbits $Bvx_S$ and $Bvx_{s_\beta(S)}$ and they must be comparable.
They are different because $s_\beta(S)=S'\cup\left\lbrace \gamma_++\beta,\gamma_--\beta\right\rbrace$ where $S'=S\setminus \left\lbrace \gamma_+,\gamma_-\right\rbrace$ and $\gamma_++\beta\neq \gamma_+,\gamma_-$.

By computing the involutions we obtain as before
$$\sigma_{vs_\beta(S)}=s_\alpha\sigma_{v(S)}s_\alpha=s_\alpha\circ \sigma_{v(S)}<\sigma_{v(S)}$$
So by Theorem \ref{frecciafacile} $Bvx_S$ is the open orbit and $\left\lbrace(v,s_\beta(S))\right\rbrace=\mathcal{E}_\alpha(v,S)$.\qedhere
\end{description}

\end{description}
\end{proof}

From this we can compute the dimension of the orbits.

\begin{Lemma}\label{dimADE}
Let $(v,S)\in V_L$.
Then $\dim Bvx_S=\#\Psi+L(\sigma_{v(S)})$.
\end{Lemma}
\begin{proof}
We know by \ref{Springer} that
\[\sigma_{v(S)}(\alpha)<0\Rightarrow\mathcal{E}_\alpha(v,S)\neq\varnothing\]
Now take an orbit $Bvx_S$ and suppose $\sigma_{v(S)}=\id$ which means $S=\varnothing$.
Then $B_v e_S=B_v.0$ is the minimal $B_v$-orbit in $\mathfrak{p}^u$ and it is easy to see that is dimension is $\#\Psi-\#\Phi^+(v)$.
It follows that $\dim Bvx_S=\#\Psi$ and $Bvx_S$ is also a minimal orbit in $G/L$.

Now suppose $L(\sigma_{v(S)})=l>0$ and fix $\alpha\in\Delta$ such that $\mathcal{E}_\alpha(v,S)\neq\varnothing$ which we know exists.
Then by induction if $(u,R)\in\mathcal{E}_\alpha(v,S)$ we have
\[\dim Bvx_S=\dim Bux_R+1=\#\Psi+L(\sigma_{u(R)})+1=\#\Psi+L(\sigma_{v(S)})\qedhere \]
\end{proof}
Note that this coincide with the dimension formula in \cite{GM}

We see that the set $V_L$ of admissible pairs for a simply laced root system doesn't depend on the characteristic of the base field.
This set already admits an order which was given by Gandini and Maffei in \cite{GM} and that we will repeat here.
\begin{Definition}\label{deford}
Let $(u,R), (v,S)\in V_L$.
We say that $(u,R)\leq (v,S)$ if and only if $\sigma_{u(R)}\leq \sigma_{v(S)}$ and $[v\sigma_S]^P\leq[u\sigma_R]^P\leq u\leq v$.
\end{Definition}
Note that the inequality $[u\sigma_R]^P\leq u$ is always true because $u(R)<0$.

From now until the end of this section, we will write $(u,S)\leq (v,R)$ for the definition above, $(u,S)\leq_2 (v,R)$ for the order induced on $V_L$ by the Bruhat order in characteristic $2$ and $(u,S)\leq_{\neq 2} (v,R)$ for the order induced on $V_L$ by the Bruhat order in characteristic different from $2$.
We have the following result:

\begin{theorem}[Theorem 1.3,\cite{GM}]\label{ordVL}
\[(u,R)\leq (v,S) \Leftrightarrow (u,R)\leq_{\neq 2} (v,S)\]
\end{theorem}

As said, the proof of this theorem makes great use of the action of the minimal parabolic subgroups on $V_L$.
To formalize this, we can define a monoid $M(W)$ generated by elements of the form $m(\alpha)$ for $\alpha\in\Delta$.
This is part of a more generic definitions for Coxeter groups $(W,\Delta)$ that can be found in \cite[3.10]{RS}.
We don't need to know most of the intricacies of this definition.
For us, it is enough to know that the monoid acts on $V_L$ and that at the level of generators the action is given by
\[m(\alpha).(v,S)=m_\alpha(v,S)\]
Note that this is visually the same definition as in the characteristic different from $2$ case, but right now we don't know if the $m_\alpha(v,S)$ coincide.
Thankfully, they do.

\begin{Lemma}\label{azioneuguale}
Suppose that we have no limitation on $\ch(\mathbb{K})$.
Fix $(v,S)\in V_L$ and $\alpha\in \Delta$.
Denote $\beta=v^{-1}(\alpha)$.
We have:
\begin{enumerate}
\item if $\sigma_{v(S)}(\alpha)<0$, then $m_\alpha (v,S)=(v,S)$.
\item if $s_\alpha v<v$, then $m_\alpha (v,S)=(v,S')$ where $S'=S\cup \left\lbrace -\beta\right\rbrace$ if $-\beta$ and $S$ are orthogonal and $S'=S$ otherwise;
\item if $v<s_\alpha v\in W^P$, then $m_\alpha (v,S)=(s_\alpha v, S')$ where $S'=S\cup \left\lbrace\beta\right\rbrace$ if $S$ and $\beta$ are orthogonal and $S'=S$ otherwise;
\item if $\beta\in\Delta_P$ and $\sigma_{v(S)}(\alpha)>0$, then $m_\alpha (v,S)=(v,S')$ where
\[
S'=\left\lbrace\begin{array}{ll}
s_\beta(S) & \text{ if } s_\beta(S)\neq S\\
\text{the representative of the }B_v\text{-orbit of } u_{-\beta}(1).e_S & \text{otherwise}
\end{array}\right.
\]
Moreover, in the last case the result doesn't depend on the characteristic of the base field,
\end{enumerate}
\end{Lemma}
\begin{proof}
\begin{enumerate}
\item
If $\ch(\mathbb{K})=2$ the result follows from Lemma \ref{Springer}.
If $\ch(\mathbb{K})\neq 2$ it follows from \cite[Lemma 7.4]{RS}.
\item
Note that this implies $-\beta\in\Phi^+(v)$.
Consider $P_\alpha=Bs_\alpha\cup BU_{-\alpha}$.
Then
\begin{align*}
P_\alpha vx_S&=Bs_\alpha vx_S\cup BU_{-\alpha}vx_S\\
&=Bs_\alpha vx_S\cup BvU_{-\beta}x_S\\
&=Bs_\alpha vx_S\cup Bvx_S\cup \bigcup_{t\in\mathbb{K}^*}Bvu_{-\beta}(t)x_S
\end{align*}

Now $Bs_\alpha vx_S\leq Bvx_S$ because $s_\alpha v <v$ and $u_{-\beta}(t)x_S=\exp(e_S+te_{-\beta})$.
If $-\beta\in S$ the last orbit is equal to $Bvx_S$ except when $t=1$ in which case is $Bvu_{-\beta}(1)x_S=Bvx_{S\setminus\left\lbrace -\beta\right\rbrace}$.
This last orbit is clearly smaller than $Bvx_S$.

If instead $-\beta\notin S$ suppose that $-\beta$ and $S$ are orthogonal.
Then it is clear that $Bvu_{-\beta}(t)x_S=Bvx_{S\cup\left\lbrace-\beta\right\rbrace}$ for every $t\in\mathbb{K}^*$ and the claim follows because $Bvx_{S\cup\left\lbrace-\beta\right\rbrace}\geq Bvx_S$.

Suppose now $-\beta\notin S$ and that there is $\gamma$ such that $-\beta-\gamma=\delta\in\Phi_P$ (note that $-\beta$ is maximal in $\Phi^+(v)$).
Then $u_\delta(s)$ acts as the identity in $S\setminus \left\lbrace \gamma\right\rbrace$ and sends $e_\gamma$ to $e_\gamma+ase_{-\beta}$ where $a\neq 0$.
It follows that for every $t$ there is $s\in\mathbb{K}$ such that $u_{\delta}(t).(e_S+te_{-\beta})=e_S$ and the claim follows.

\item 

Note that $\beta$ is minimal in $\Psi\setminus \Phi^+(v)$ and maximal in $\Phi^+(s_\alpha v)$.
Consider $P_\alpha=B\cup Bs_\alpha U_\alpha$
Then
\begin{align*}
P_\alpha vx_S&=Bvx_S\cup Bs_\alpha U_{\alpha}vx_S\\
&=Bvx_S\cup Bs_\alpha v x_S\cup\bigcup_{t\in\mathbb{K}^*}Bs_\alpha u_{\alpha}(t) vx_S\\
&=Bvx_S\cup Bvs_\alpha x_S\cup \bigcup_{t\in\mathbb{K}^*}Bs_\alpha vu_{\beta}(t)x_S
\end{align*}

We have $s_\alpha v>v$ and $s_\alpha v\in W^P$ so $Bvx_S$ can't be the open orbit.

As before $u_{\beta}(t)x_S=\exp(e_S+te_{\beta})$.
By reasoning as in the point above, the claim follows.

\item 

We have $s_\alpha v=vs_\beta$ so $[s_\alpha v]^P=v$ where $[s_\alpha v]$ is the representative in $W^P$ of the coset $s_\alpha vW_P$ in $W^P$.

Write $P_\alpha=Bs_\alpha\cup BU_{-\alpha}$ 
\begin{align*}
P_\alpha vx_S&=Bs_\alpha vx_S\cup BvU_{-\beta}x_S\\
&=Bvx_{s_\beta(S)}\cup Bvx_S\cup \bigcup_{t\in\mathbb{K}^*}Bv\exp(u_{-\beta}(t).e_S)
\end{align*}
where we used the facts that $s_\beta$ can be represented by an element in $L$ and $U_{-\beta}\subseteq L$.

Suppose $s_\beta(S)\neq S$.
Then $\sigma_{vs_\beta(S)}>\sigma_{v(S)}$ and by Lemma \ref{dimADE} $Bvx_{s_\beta(S)}$ must be the open orbit in $P_\alpha vx_S$ given that its dimension is higher then the dimension of $Bvx_S$.

Hence, suppose $s_\beta(S)=S$.
As a first thing, note that the support of $u_{-\beta}(1).e_S$ is $S$ if there is no $\gamma\in S$ such that $\gamma-\beta\in\Phi$ and it is $S\cup \left\lbrace \gamma-\beta\right\rbrace$ otherwise.
We claim that, in this last case, the support uniquely determines the orthogonal subsets $S'$ that parametrizes the $B_v$-orbit.
In particular, the result is independent from the characteristic of $\mathbb{K}$.

To see this note that $u_{-\beta}(t).e_S=e_S+ke_{\gamma-\beta}$ where $k\in\mathbb{K}^*$ depends on $t$ and our choices of base vectors in the root spaces.
Now, $\gamma$ is certainly not orthogonal to $\gamma-\beta$ and there is at most one other root $\delta\in S$ that is not orthogonal to $\gamma -\beta$.
In this last case, it must be $\delta+\beta\in\Phi$, so $(\delta,\beta)<0$ and $(\gamma -\beta,\delta)>0$.
Suppose at first that such $\delta$ doesn't exist and let $u_\beta(s)$ act on $e_S+ke_{\gamma-\beta}$.
It must be $u_\beta(s).\left(e_S\right)=e_S$, so there is $s\in\mathbb{K}^*$ such that $u_\beta(s).\left(e_S+ke_{\gamma-\beta}\right)=e_{S\setminus\left\lbrace\gamma\right\rbrace}+ke_{\gamma-\beta}$.
The set $\left(S\setminus\left\lbrace\gamma\right\rbrace\right)\cup \left\lbrace \gamma-\beta\right\rbrace$ is orthogonal and the claim follows.

Suppose then that such a $\delta$ is in $S$.
Then $\tau=\gamma-\beta-\delta\in\Phi_P$.
If $\tau$ is positive $U_\tau\subseteq B$, so we can act with $u_\tau(s)$ on $e_S+ke_{\gamma-\beta}$ without changing the $B$-orbit.
Note that $\tau$ is orthogonal to all roots in $S$ except $\delta$ and $\gamma$.
Moreover, clearly $\delta+\tau=\gamma-\beta\in\Phi^+$, so $\tau$ can be added to $\tau$, which implies that it cannot be added to $\gamma$.
It follows that $u_\tau(s)$ acts as the identity on the root spaces of every root in $S\setminus\left\lbrace\delta\right\rbrace$.
It is then easy to see that there is an $s\in\mathbb{K}^*$ such that $u_\tau(s).\left(e_S+ke_{\gamma-\beta}\right)=e_S$.

Suppose at last $\tau$ negative, so $-\tau$ positive.
We still have $s\in\mathbb{K}^*$ such that $u_\beta(s).\left(e_\gamma+ke_{\gamma-\beta}\right)=k e_{\gamma-\beta}$, but for such an $s$ it must be $u_\beta(s).e_{\delta}=e_\delta+k'e_{\delta+\beta}$ for some $k'\in\mathbb{K}^*$.
Then $u_\beta(s).\left(e_S+ke_{\gamma-\beta}\right)=e_{S\setminus\left\lbrace\gamma\right\rbrace}+ke_{\gamma-\beta}+k'e_{\delta+\beta}$.
If we now let $u_{-\tau}(r)$ act on $e_{S\setminus\left\lbrace\gamma\right\rbrace}+ke_{\gamma-\beta}+k'e_{\delta+\beta}$ we see that it is actually the identity on $e_{S\setminus\left\lbrace\gamma\right\rbrace}$ and even on $k'e_{\delta+\beta}$ because $\tau$ can be added to $\delta+\beta$, so it can't be subtracted.
We then can easily find $r\in\mathbb{K}^*$ such that 
\[u_{-\tau}(r)u_{\gamma}(s).\left(e_S+ke_{\gamma-\beta}\right)=e_{S'}+ke_{\gamma-\beta}+k'e_{\delta+\beta}\]
where $S'=S\setminus\left\lbrace \gamma,\delta\right\rbrace$.
The set $S'\cup\left\lbrace \gamma-\beta,\delta+\beta\right\rbrace$ is orthogonal because 
\[(\gamma-\beta,\delta+\beta)=\underbrace{(\gamma,\delta)}_0+\underbrace{(-\beta,\delta)}_1+\underbrace{(\gamma,\beta)}_1-\underbrace{(\beta,\beta)}_2=0\]
This completes the proof.\qedhere
\end{enumerate}
\end{proof}
This lemma clearly proves that the value of $m_\alpha(v,S)$ doesn't depend on the base field.

We can also define a \textit{length function} $l\colon V_L\longrightarrow \mathbb{N}$ as $l(v,S)=\dim Bvx_S-d$ where $d=\min_{(v,S)\in V_L}\dim Bvx_S$.
By Lemma \ref{dimADE} and \cite[Formula 1]{GM}, this definition doesn't depend on the characteristic of the base field.

To conclude, we need a definition from \cite{RS}.
Let $\preceq$ be an order on $V_L$.
\begin{Definition}[One-step property]
Let $x\in V_L$ and $\alpha\in\Delta$ such that $m(\alpha).x\neq x$.
Then $y\preceq m(\alpha).x$ if and only if at least one of the following is true:
\begin{enumerate}
\item $y\preceq x$;
\item there is $z$ such that $m(\alpha).z=m(\alpha).y$ and $z\preceq x$;
\end{enumerate}
\end{Definition}

We also need this important result.
\begin{theorem}\label{stord}
Let $\preceq$ be an order on $V_L$ such that:
\begin{enumerate}
\item $x\preceq m(\alpha).x$;
\item if $x\preceq y$ then $m(\alpha).x\preceq m(\alpha).y$;
\item if $x\preceq y$ and $l(y)\leq l(x)$, then $x=y$
\end{enumerate}
In this case we say that $\preceq$ agrees with the action of $\Delta$.
Suppose also that $\preceq$ has the one-step property.
Then 
\[(u,R)\preceq (v,S)\Leftrightarrow (u,R)\leq (v,S)\]
\end{theorem}
\begin{proof}
By Section 6 of \cite{RS} the order $\preceq$ coincides with what Richardson and Springer call the \textit{standard order}, which in turn (see \cite[Theorem 7.11]{RS}) coincides with the order $\leq_{\neq 2}$ on $V_L$.
We use Theorem \ref{ordVL} to conclude.
\end{proof}

Finally, we need a general result from \cite{Conjugacy}.
\begin{Lemma}[Lemma 2, \cite{Conjugacy}]\label{conj}
Let $G$ act on a variety $V$. 
Suppose $H\subseteq G$ is closed and $U\subseteq V$ be a closed subset of $V$ invariant under $H$.
If $G/H$ is complete, then $G.U$ is closed.
\end{Lemma}

We can finally obtain the characterization of the Bruhat order in $G/L$ we were looking for:
\begin{theorem}\label{2ordine}
Let $(u,R),(v,S)\in V_L$.
Then
\[(u,R)\leq (v,S) \Leftrightarrow (u,R)\leq_2(v,S)\]
\end{theorem}
\begin{proof}
We want to prove the correspondence with Theorem \ref{stord}.
Then, we need to show that $\leq_2$ agrees with the action of $\Delta$ and that it has the one-step property.
The first part is clear, because the action of $\Delta$ is exactly the action of the minimal parabolic subgroups and the length function $l$ is the dimension up to a constant.
We need to show that $\leq_2$, which is the Bruhat order, verifies the one-step property.
From Theorem \ref{conj} with $G=P_\alpha$ and $H=B$ we obtain that $P_\alpha \cdr{\mathcal{O}}$ is closed.
But $\cdr{\mathcal{O}}=\bigcup_{\mathcal{O}'\leq \mathcal{O}}\mathcal{O}'$, so
\[\cdr{P_\alpha\mathcal{O}}=P_\alpha\cdr{\mathcal{O}}=\bigcup_{\mathcal{O}'\leq \mathcal{O}}P_\alpha\mathcal{O}'\]
Then take $(v,S)$ and $\alpha$ such that $m_\alpha(v,S)\neq (v,S)$ and fix $(u,R)\leq_2 m_\alpha(v,S)$.
This means $Bux_R\subseteq \cdr{P_\alpha vx_S}$.
By what we said above this implies that there is $Bv'x_{S'}\subseteq \cdr{Bvx_S}$ such that $Bux_R\subseteq P_\alpha v'x_{S'}$.
Hence, $m_\alpha(u,R)=m_\alpha(v',S')$.
This proves the one-step property and the theorem.
\end{proof}

It follows that the characterization of the Bruhat order doesn't depend on the characteristic.
The final result is the following:
\begin{theorem}\label{ordineADE}
In the simply laced case we have
\[Bux_R\leq Bvx_S\Leftrightarrow \sigma_{u(R)}\leq\sigma_{v(S)}\text{ and }[v\sigma_S]^P\leq[u\sigma_R]^P\leq u\leq v\]
In particular, the Bruhat order in the simply laced case doesn't depend on the characteristic of the base field.
\end{theorem}
\begin{proof}
By Theorem \ref{2ordine}
\[(u,R)\leq (v,S) \Leftrightarrow (u,R)\leq_2(v,S)\]
and by Theorem \ref{ordVL}
\[(u,R)\leq (v,S) \Leftrightarrow (u,R)\leq_{\neq 2}(v,S)\]
Hence the two orders coincide.
The claim follows from Definition \ref{deford}.
\end{proof}

\section{The parametrization in the type B case}\label{B}

Let $\mathbb{K}$ be an algebraically closed field of characteristic $2$ and $G$ a connected, reductive, linear algebraic group over $\mathbb{K}$ with a type $\bf{B}$ root system.

We can define a realization of this root system in the following way:
take in $\mathbb{R}^n$ the sets of vectors $\left\lbrace \pm e_1,\ldots,\pm e_n\right\rbrace_{1\leq i\leq n}$ and $\left\lbrace \pm e_i\pm e_j\right\rbrace_{1\leq i,j\leq n}$ where $e_i$ is the canonical base.
We may choose as a base $\left\lbrace e_1-e_{2},\ldots,e_{n-1}-e_n,e_n\right\rbrace$.
The highest root $\theta$ is $e_1+e_2=(e_1-e_2)+2(e_2-e_3)+\cdots+2e_n$.
It follows that $\alpha_P=e_1-e_2$ and $\Delta_P$ is obtained by omitting $e_1-e_2$, hence $\Phi_P=\left\lbrace e_i\pm e_j\right\rbrace_{i,j\neq 1}\cup\left\lbrace e_i\right\rbrace{i\neq 1}$.
On the other hand
\[\Psi=\left\lbrace e_1\pm e_i\right\rbrace_{2\leq i\leq n}\cup \left\lbrace e_1\right\rbrace\]

Note that roots in $\Psi$ are always comparable and that there is a unique short root $\alpha_0=e_1$.
Moreover, no root in $\Psi$ is orthogonal to $\alpha_0$ while if $\alpha$ is long then there is a unique root $\alpha^\perp$ that is orthogonal to $\alpha$. 
If $\alpha<\alpha^\perp$ it must be $\alpha<\alpha_0<\alpha^\perp$.
To simplify the notation we put $\alpha_0^\perp=\alpha_0$.

Thanks to Lemma \ref{IET} we know that if $\alpha_0$ is the short root in $\Psi$ and $\beta\in \Phi_P$, then
$$u_\beta(t).e_{\alpha_0}=e_{\alpha_0} \text{ for every }t\in\mathbb{K}$$
If, instead, $\alpha\in \Psi$, $\alpha\neq \alpha_0$ and $\beta\in\Phi$ then
$$u_\beta(t).e_\alpha=e_\alpha+a_{\beta,\alpha}te_{\alpha+\beta}+b_{\beta,\alpha}t^2e_{\alpha+2\beta}$$
where $a_{\beta,\gamma}\neq 0$ (respectively, $b_{\beta,\gamma}\neq 0$) if and only if $\alpha+\beta\in\Phi$(respectively, $\alpha+2\beta\in\Phi$).

Recall that if $x=\sum_{\alpha\in \Psi}a_\alpha e_\alpha\in \mathfrak{p}^u$ the set $\supp(x)=\left\lbrace \alpha\in\Psi\mid a_\alpha\neq 0\right\rbrace$ is called the \textit{support} of $x$.
If $M$ is a subset of $\mathfrak{p}^u$, the \textit{support} of $M$ is the union of $\supp(x)$ for every $x\in M$.
Moreover, if $\beta\in\Psi$ we will denote with $\Psi_{\beta}$ the set $\left\lbrace \alpha\in\Psi\mid \alpha\leq \beta\right\rbrace$.

\begin{theorem}\label{parametrizzazioneB}
Fix $v\in W^P$ and consider the family
$$\mathcal{H}_v=\left\lbrace S\subseteq \Phi^+(v)\mid S \text{ is orthogonal or } S=\left\lbrace \alpha_0, \alpha_0+\gamma\right\rbrace \text{ with }\alpha_0 \text{ short and }\gamma\in\Phi^+_P\right\rbrace$$
Then there is a bijection
\begin{align*}
\mathcal{H}_v&\longleftrightarrow \left\lbrace B_v\text{-orbits in }\mathfrak{p}_u\right\rbrace\\
S&\longleftrightarrow B_ve_S
\end{align*}
\end{theorem}
\begin{proof}
We will first show that the map is surjective.
Fix an orbit $\mathcal{O}$ and an element $x\in\mathcal{O}$.
We can suppose, by acting with $\mathfrak{u}_v\subseteq B_v$ if needed, that $S=\supp(x)\subseteq \Phi^+(v)$ (see Equation \ref{action}).
Now, $S$ has a minimal root $\alpha$.
If $S=\left\lbrace \alpha\right\rbrace$ we are done, because if $x=x_\alpha=te_\alpha$ for some $t\in\mathbb{K}^*$, then $e_\alpha\in T.x_\alpha$.
The same is true if $S$ is orthogonal, because orthogonal roots are independent.

Suppose $\#S\neq 1$, $S$ not orthogonal and write $x=\sum_{\alpha\in S}a_\alpha\alpha$

If $\alpha$ is long, consider $\beta$ minimal in $\left(S\setminus \left\lbrace \alpha,\alpha^\perp\right\rbrace\right)$.
We want to find an element $x'\in\mathcal{O}$ such that $S'=\supp(x')\subseteq \Phi^+(v)$ and $S'\cap \Psi_\beta=\left(S\cap\Psi_\beta\right)\setminus \left\lbrace \beta\right\rbrace$.
Given that the roots in $\Psi$ are a finite number, this implies inductively that we can find $x\in\mathcal{O}$ such that $\supp(x)$ is orthogonal.

Now, $\beta\neq \alpha^\perp$, hence $\langle\beta,\alpha\rangle\neq 0$.
Then we know that $\gamma=\beta-\alpha\in\Phi^+$ because $\beta>\alpha$ and by the characterization of $\Psi$ and $\Phi_P$ it must be $\gamma\in\Phi_P^+$.
It follows that the one-parameter subgroup $U_\gamma$ is contained in $B_L\subseteq B_v$ and we must have $u_\gamma(t).e_\alpha=e_\alpha+ate_\beta$ or $u_\gamma(t).e_\alpha=e_\alpha+ate_\beta+bt^2e_{\alpha+2\gamma}$ where in both cases $a\neq 0$.
Note that $u_\gamma(t)$ act as the identity on $\alpha^\perp$ and all the other roots in $S$ are greater than $\beta$, hence in both cases the coefficient of $e_\beta$ in $u_\gamma(t).x$ is $a_\beta+at$.
It follows that there is $t_0\in\mathbb{K}^*$ such that $\beta\notin\supp(u_\gamma(t_0).x)$ and, to be more precise 
\[\supp(u_\gamma(t_0).x)\cap\Psi_{\beta}=\left(\supp(x)\cap \Psi_{\beta}\right)\setminus\left\lbrace\beta\right\rbrace\]
Again,  we can act on $u_\gamma(t_0).x$ with elements of $\mathfrak{u}_v\subseteq B_v$ to obtain an element $x'\in\mathcal{O}$ such that $\supp(x')= \supp(u_\gamma(t_0).x)\cap \Phi^+(v)$ and that's the element we were looking for.

If instead $\alpha=\alpha_0$ is short, take $\alpha'$ minimal in $S\setminus\left\lbrace \alpha\right\rbrace$.
If $S=\left\lbrace\alpha,\alpha'\right\rbrace$ we are done (they are linearly independent), so suppose $S\setminus\left\lbrace \alpha,\alpha'\right\rbrace\neq\varnothing$ and take $\beta$ minimal in it.
As before, we want to find $x'\in\mathcal{O}$ such that $S'=\supp(x')\subseteq \Phi^+(v)$ and $S'\cap \Psi_\beta=\left(S\cap\Psi_\beta\right)\setminus \left\lbrace \beta\right\rbrace$.

To do this, note that $\langle \beta,\alpha'\rangle\neq 0$ because $\alpha'^\perp<\alpha_0<\alpha'<\beta$, hence $\beta\neq \alpha'^\perp$.
Then, there is $\gamma=\beta-\alpha'\in\Phi^+_P$, hence $U_\gamma\subseteq B_L\subseteq B_v$and we have
\[u_\gamma(t).e_{\alpha'}=e_{\alpha'}+ate_\beta\]
while $u_\gamma(t).e_\alpha=e_\alpha$ by Lemma \ref{IET}.
As before, it follows that there is $t_0\in\mathbb{K}^*$ such that
\[\supp(u_\gamma(t_0).x)\cap \Psi_{\beta}=\left\lbrace\alpha,\alpha'\right\rbrace\]
and we can suppose $\supp(u_\gamma(t_0).x)\subseteq \Phi^+(v)$.
This concludes the proof of surjectivity.

We will now show the injectivity.
Suppose $S,R\in\mathcal{H}_v$ and $\mathcal{O}=B_ve_S=B_ve_R$.
Note that it must be 
$$\min S=\min R=\min \supp(\mathcal{O})$$
Denote $\alpha=\min \supp(\mathcal{O})$.
If $\alpha$ is not the short root, then both $S$ and $R$ must be orthogonal and, by Theorem \ref{orthogonalsubset}, $B\sigma_{v(S)}B=B\sigma_{v(R)}B$, so $\sigma_{v(S)}=\sigma_{v(R)}$ and $S=R$.

Now suppose $\alpha$ short and $\beta=\min\left\lbrace\supp(\mathcal{O})\setminus \left\lbrace \alpha\right\rbrace\right\rbrace$.
If $\beta\notin \Phi^+(v)$, then it is clear that $S=R=\left\lbrace\alpha\right\rbrace$, hence suppose $\beta\in\Phi^+(v)$.
Now, $\beta\notin\supp(Be_{\alpha})$ because $\spa(e_\alpha)$ is $B_L$-stable by Lemma \ref{IET}.
It follows that $\beta\in S$ and $\beta\in R$, so it must be 
\[S=\left\lbrace\alpha,\beta\right\rbrace =R\qedhere\]
\end{proof}
\section{The parametrization in the type C case}\label{tipoC}

Let $\mathbb{K}$ be an algebraically closed field of characteristic 2. Suppose that $G$ is a connected, reductive, linear algebraic group over $\mathbb{K}$ with a type $\bf{C}$ root system.
Fix $T\subseteq B$ a maximal torus and a Borel subgroup.
For our examples, we will take $G=\matr{Sp}(2n,\mathbb{K})=\left\lbrace M\in \matr{SL}(2n,\mathbb{K})\mid M^t\Omega M=\Omega\right\rbrace$ where
\[\begin{array}{lcr}
\Omega=\left(\begin{array}{c|c}
0 & L\\
\hline
-L& 0
\end{array}\right) & & L=\left(\begin{array}{ccc}
& & 1\\
& \reflectbox{$\ddots$} &\\
1& &
\end{array}\right)
\end{array}
\]
Note that in characteristic $2$ we have $-L=L$.
We choose as a maximal torus $T$ the diagonal matrices in $G$ and as Borel subgroup $B$ the upper triangular matrices in $G$.
With these choices, the only parabolic subgroup that verifies the hypothesis is 
\[P=\left\lbrace \left(\begin{array}{c|c}
						A & B\\
						\hline
						0 & C
						\end{array}\right)\in \matr{Sp}(2n,\mathbb{K})
						\mid A,B,C\in M(n,\mathbb{K})
\right\rbrace
 \]

We will make use of the following realization of the root system in $\mathbb{R}^n$

$$\Phi=\left\lbrace \pm e_i\pm e_j \mid i \neq j\right\rbrace \cup \left\lbrace \pm 2e_i \right\rbrace $$

where $e_1, \ldots, e_n$ is the canonical base of $\mathbb{R}^n$. As a (ordered) base of such root system we choose the set $\Delta=\left(e_1-e_2, \ldots, e_{n-1}-e_n, 2e_n\right)$.

If we label the roots of $\Delta$ as $\left(\alpha_1, \ldots, \alpha_{n-1},\alpha_n \right)$ the longest root is $2\alpha_1+2\alpha_2+\cdots+2\alpha_{n-1}+\alpha_n$. 
It follows that 
\[\Psi=\left\lbrace e_i+e_j\right\rbrace_{1\leq i<j\leq n}\cup\left\lbrace 2e_i\right\rbrace_{1\leq i\leq n}\]

On the other hand, \[
\Phi_P=\left\lbrace e_i-e_j\right\rbrace_{i,j=1,\ldots,n}\]
Note that $\Phi_P$ is of type $\bf{A}_{n}$.

As before we have that if $u_\alpha(t)$ is the one-parameter subgroup of $\alpha$, then 
$$u_\alpha(t).x_\beta=x_\beta + atx_{\beta + \alpha} + bt^2x_{\beta+2\alpha}$$
 for some $a,b\in\mathbb{F}$ and $a=0$ if and only if $\beta+\alpha \notin \Phi$ or both $\beta+\alpha \in \Phi$ and $\beta-\alpha \in \Phi$. 
 Note that in both cases $b=0$, too.

 Moreover, note that if $\beta$ is long and $\alpha$ is short and there is $\gamma\in\Phi_P^+$ with $\alpha+\gamma=\beta$, then also $\alpha-\gamma\in \Phi$.
In fact, $s_{\alpha}(\alpha+\gamma)=\alpha+\gamma-\langle \alpha, \alpha+\gamma\rangle \alpha$
and $\langle \alpha, \alpha+\gamma\rangle=2$, so $\gamma-\alpha=s_\alpha(\beta)\in \Phi$.
This  implies $(\alpha,\gamma)=0$. 

In our example, the algebra $\mathfrak{p}^u$ is the set 
\[\mathfrak{p}^u=\left\lbrace\left(\begin{array}{c|c}
					0 & M\\
					\hline
					0 & 0
					\end{array}\right)\in \mathfrak{sp}(2n,\mathbb{K})\mid M\in M(n,\mathbb{K})\right\rbrace\]
It is easy to see that $\left(\begin{array}{c|c}
					0 & M\\
					\hline
					0 & 0
					\end{array}\right)\in \mathfrak{sp}(2n,\mathbb{K})$ if and only if $M$ is symmetric with respect to the anti-diagonal, that is $M_{i,j}=M_{n-j+1,n-i+1}$.
For this reason, we will represent the matrix $M$ with only the upper-left entries in the following way (here is n=5)
\[
\Yvcentermath1
M=\gyoung(1;1;0;1;0,0;0;1;1,1;1;0,1;1,1)\]
Every root space is generated by an element with $0\in\mathbb{K}$ in every square but one.
In particular, if we number the row and the column starting from the upper right vertex
\begin{equation}\label{numbering}
\Yvcentermath1
M=\gyoung(:5:4:3:2:1:,0;1;0;0;0:1,0;0;0;0::2,0;0;0:::3,0;0::::4,0:::::5)
\end{equation}

then the root space whose generator has a $1\in\mathbb{K}$ in row $i$ and column $j$ with respect to this numbering is relative to the root $e_i+e_j$ while if the $1$ is in position $(i,i)$ it corresponds to the root $2e_i$.
For simplicity, from now on an empty square will mean a square with the zero element of $\mathbb{K}$ while a $\bullet$ will mean every element different from the zero element.
The previous matrix will then become
\[
\Yvcentermath1
M=\gyoung(;;\bullet;;;,;;;;,;;;,;;,;)\]
Note that this notation shouldn't generate ambiguity because in our examples the roots will be linearly independent, so the $B$-orbit will not depend on the specific values of $\bullet$.

Given that there are one to one correspondences between roots in $\Psi$, root spaces in $\mathfrak{p}^u$ and generators of root spaces up to multiplication by scalars, we will often denote a root with the correspondent diagram, for example if $M$ is the diagram above, then $M=e_1+e_4$.

Now, fix a root $\alpha\in\Psi$.
With the notation above, the roots smaller than $\alpha$ are on the lower left (symbol $\cdot$ in the first diagram) while the roots that are bigger are on the upper right (symbol $\circ$ in the first diagram).
Moreover, the roots that are not orthogonal to $\alpha$ are exactly the roots that share a row or a column with $\alpha$  with a caveat: the row and column are "reflected" by the antidiagonal.
What we mean is that if we label the rows and columns as in Diagram \ref{numbering} and $\alpha$ is in the $i$-th row and $j$-th column, the roots that are not orthogonal to $\alpha$ will be exactly the roots in row $i$ and $j$ and in column $i$ and $j$ (symbol $\star$ in the second diagram).

\[
\begin{array}{lccr}
\gyoung(;;\circ\circ\circ\circ,\cdot\bullet\circ\circ,\cdot\cdot;,\cdot\cdot,\cdot) & & &\gyoung(;;\star;;\star;,\star\bullet\star\star,;;\star;,\star\star,;) 
\end{array}
\]
To keep the diagrams simple, we will always make our examples by fixing $v=\omega^P$ the longest element in $W^P$.

A first result that we can prove thanks to the realization is the next one.
\begin{Proposition}\label{cortelunghe}
Let $S\subseteq \Psi$ be a subset of short roots.
Then for every $b\in B$ the support $\supp (b.e_S)$ contains only short roots.
\end{Proposition}
\begin{proof}
Note that if $S=\left\lbrace \beta_1,\ldots,\beta_n\right\rbrace$, then $\supp(b.e_S)\subseteq \bigcup_{i=1}^n\supp(b.e_{\beta_i})$.
It follows that we can suppose that $S=\left\lbrace\beta\right\rbrace$ and $\beta=e_i+e_j$ with $i<j$.

We can write $b=tx_{\alpha_1}\cdots x_{\alpha_n}$ with $t\in T$ and $x_{\alpha_i}\in U_{\alpha_i}$.
Note that the action of $t$ doesn't change the support, so we can suppose $t=\id$.
We will prove the claim by induction on $n$.

Suppose $n=1$, then $b=u_\alpha(t)$.
But now $\beta+\alpha$ is a long root if and only if $\alpha=e_i-e_j$ and $\beta-\alpha=2e_j\in \Phi$, so $b.e_\beta=e_\beta$ by Lemma \ref{IET}.
Now suppose $n>1$ and consider $b'=x_{\alpha_2}\cdots x_{\alpha_{n}}$.
By induction we have that $\supp(b'.e_S)$ contains no long root and for every (short) root $\beta'\in\supp(b'.e_S)$ we have that $\supp(x_{\alpha_1}.e_{\beta'})$ contains no long root.
It follows that $\supp(b.e_S)$ contains only short roots as well.
\end{proof}

For every short root $\alpha=e_i+e_j\in\Psi$ there are only two long roots in $\Psi$ that are not orthogonal with $\alpha$, namely $2e_i$ and $2e_j$.
Of them, one is bigger and one is smaller than $\alpha$.
We will denote the former with $\suc(\alpha)$.
In the cases where $\alpha$ is long we will define $\suc(\alpha)=\alpha$.

Now, fix an element $v\in W^P$. 
We will now define a family of representatives for the $B_v$-orbits on $\mathfrak{p}_u$.

\begin{Definition}\label{admissibleC}
Let $\Phi^+(v)=\left\lbrace \alpha \in \Psi \mid v(\alpha)<0\right\rbrace$ and $S\subseteq \Phi^+(v)$. Then $S$ is \textit{admissible} (for $v$) if $S$ can be partitioned as $X(S)\sqcup Z(S)$ where:
\begin{enumerate}[label=(\roman*)]
\item\label{1} $X(S)$ is orthogonal;
\item\label{2} every element of $Z(S)$ is a long root $\beta$ and for every $\beta\in Z(S)$ exists a $\alpha$ in $X(S)$ and $\gamma\in\Phi^+_P$ verifying $\beta=\alpha+\gamma$. This element is unique, so define $\pad(\beta)=\alpha$.
\end{enumerate}
\end{Definition}

Note that in point \ref{2} of Definition \ref{admissibleC} $\alpha$ and $\gamma$ are short and we have $\langle\alpha,\gamma\rangle=0$ while $\langle\beta,\alpha\rangle\neq 0$ and $\langle\beta,\gamma\rangle\neq 0$.
We are exactly in the situation where $\gamma$ can be both added and subtracted to $\alpha$, hence while $\alpha+\gamma$ is a (long) root, the action of $U_\gamma$ on $e_\alpha$ is the identity.

We may wonder about the uniqueness of the elements introduced in Definition \ref{admissibleC}.
It is clear that $X(S)=S_s\cup\left\lbrace \beta\in S_l\mid \beta \text{ and }S_s \text{ are orthogonal}\right\rbrace$, hence the partition is unique.
Regarding the uniqueness in point \ref{2}, if $\beta\in Z(S)$ suppose that there are $\tau,\delta\in X(S)$ and $\gamma_\tau,\gamma_\delta\in \Phi^+_P$ such that $\gamma_\tau+\tau=\gamma_\delta+\delta=\beta$.
Then, by Lemma \ref{sommaradici}, $\gamma_\tau$ and $\delta$ can't be added nor subtracted, hence $\langle\gamma_\tau,\delta\rangle=0$.
But $\langle \tau,\delta\rangle=\langle \beta-\gamma_\tau,\delta\rangle=\underbrace{\langle \beta,\delta\rangle}_{\neq 0}-\underbrace{\langle\gamma_\tau,\delta\rangle}_0\neq 0$ and that's a contradiction.

There is another important thing to note about this definition.
Take $S$ admissible and $\beta\notin S$ long.
Then $S\cup\left\lbrace\beta\right\rbrace$ is still admissible if and only if for every $\alpha\in S$, $\langle\alpha,\beta\rangle\neq 0$ implies $\alpha<\beta$.
In particular, $S\cup\left\lbrace\beta\right\rbrace$ is admissible whenever $\beta\nless\alpha$ for every $\alpha\in S$.

Now analyse the following diagrams where we denoted with $X_i$ the roots that are in $X(S)$ and with $Z_i$ the roots in $Z(S)$.
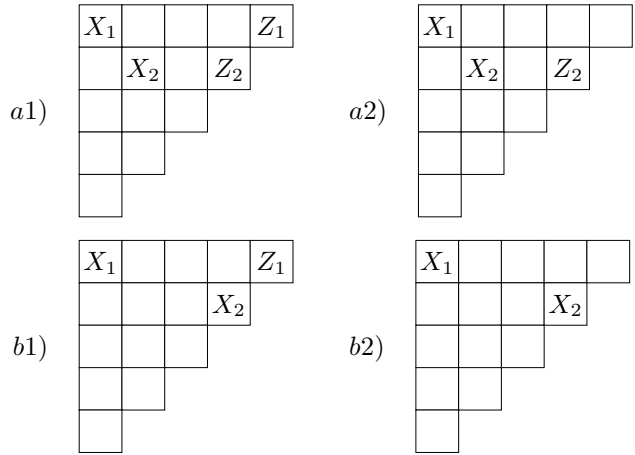
\[
\Yboxdim{16pt}
\Yvcentermath1
\begin{array}{rlcrl}
a1) & \gyoung(<X_1>;;;;{Z_1},;;{X_2};;{Z_2},;;;,;;,;)& & a2) & \gyoung(<X_1>;;;;,;;{X_2};;{Z_2},;;;,;;,;)
\end{array}
\]
\[
\Yboxdim{16pt}
\Yvcentermath1
\begin{array}{rlcrl}
b1) & \gyoung(<X_1>;;;;{Z_1},;;;;{X_2},;;;,;;,;)&  & b2) & \gyoung(<X_1>;;;;,;;;;{X_2},;;;,;;,;)
\end{array}\]

It is easy to see that they are all admissible.
Moreover $a1)$ and $a2)$ generate the same $B_v$-orbit and the same can be said for $b1)$ and $b2)$.
To see this we will show that there is $b\in B_L$ that sends $a2)$ to $a1)$ .
Recall that with our conventions we have:
\[
\begin{array}{cccc}
X_1=e_1+e_5 & Z_1=2e_1 & X_2=e_2+e_4 & Z_2=2e_2
\end{array}
\]
Then $u_{e_1-e_2}(t)$ acts as the identity on $X_1$ and sends $e_{\left\lbrace X_2,Z_2\right\rbrace}$ to $e_{\left\lbrace X_2,Z_2\right\rbrace}+ate_{e_1+e_4}+bte_{e_1+e_2}+ct^2e_{2e_1}$ for fixed non-zero $a,b,c\in\mathbb{K}^*$.
Now the subgroups $u_{e_i-e_5}(s)$ for $i=2,4$ act as the identity on all the roots except $X_1$ and we can use it to delete the components along $e_{e_1+e_4}$ and $e_{e_1+e_2}$.

Schematically
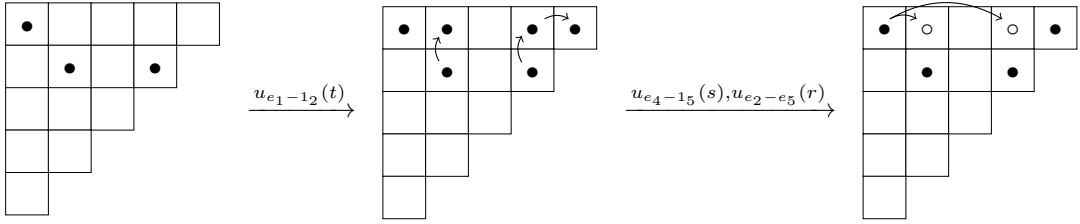
\[\Yvcentermath1
\Yboxdim{16pt}
\begin{array}{rcccl}
\gyoung(\bullet;;;;,;;\bullet;;\bullet,;;;,;;,;) & \xrightarrow{u_{e_1-1_2}(t)} & \begin{tikzpicture}[baseline=-6ex]\Yvcentermath1
\tgyoung(0cm,0cm,\bullet;\bullet;;\bullet;\bullet,;;\bullet;;\bullet,;;;,;;,;)
\draw[->] (21pt,-4.5pt) to [bend left] (21pt,4.5pt);
\draw[->] (52pt,-5pt) to [bend left] (52pt,5pt);
\draw[->] (60.5pt,0.4cm) to [bend left] (69.5pt,0.4cm); 

\end{tikzpicture}& \xrightarrow{u_{e_4-1_5}(s),u_{e_2-e_5}(r)} & \begin{tikzpicture}[baseline=-6ex]\Yvcentermath1
\tgyoung(0cm,0cm,\bullet;\circ;;\circ;\bullet,;;\bullet;;\bullet,;;;,;;,;)
\draw[->] (10pt,0.4cm) to [bend left] (52pt,0.4cm);
\draw[->] (10pt,0.4cm) to [bend left] (20pt,0.4cm);
\end{tikzpicture}

\end{array}\]

For $b1)$ and $b2)$ the reasoning is similar.
It follows that to get the uniqueness we need to be stricter.

\begin{Definition}\label{fulladmissibleC}
Let $S=X(S)\sqcup Z(S)$ be admissible for $v$.
Then $S$ is \emph{full admissible} if:

\begin{enumerate}
\item for every $\gamma\in X(S)$ long and $\alpha\in S$ short such that $\suc(\alpha)\in\Phi^+(v)$, $\suc(\alpha)>\gamma$ and $\alpha\ngtr\gamma$ we have $\suc(\alpha)\in Z(S)$;
\item for every $\gamma\in Z(S)$ and $\alpha\in S$ short such that $\suc(\alpha)\in\Phi^+(v)$, $\suc(\alpha)>\gamma$ and $\alpha\ngtr \pad(\gamma)$ we have $\suc(\alpha)\in Z(S)$.
\end{enumerate}
\end{Definition}

It is easy to see that $a1)$ and $b1)$ verify Definition \ref{fulladmissibleC} while $a2)$ and $b2)$ don't.
Actually, we can obtain $a1)$ from $a2)$ and $b1)$ from $b2)$ through the following process.

\begin{Definition}\label{fac}
Let $S=X(S)\sqcup Z(S)$ be admissible for $v$.
Denote
\begin{align*}
A(S)&=\left\lbrace \beta\in\Phi^+(v)\setminus S\mid \exists\alpha\in S_s, \beta=\suc(\alpha) \text{ and }\exists\gamma\in X(S) \text{ long}, \beta>\gamma \text{ and } \alpha\ngtr\gamma\right\rbrace\\
B(S)&=\left\lbrace \beta\in\Phi^+(v)\setminus S\mid \exists \alpha\in S_s,\beta=\suc(\alpha)\text{ and }\exists\gamma\in Z(S), \beta>\gamma\text{ and }\alpha\ngtr\pad(\gamma)\right\rbrace
\end{align*}
Then the set
\[\cdr{S}=S\cup A(S)\cup B(S)\]
is called the \textit{full admissible completion} of $S$.
\end{Definition}
Note that $S$ is full admissible if and only if $A(S)=B(S)=\varnothing$.

We want to show that the full admissible completion of an admissible set $S$ is indeed full admissible.
Note that $A(S),B(S)\subseteq \suc(S_s)$, hence the roots in $A(S)$ and $B(S)$ are not orthogonal to exactly one root in $S$.
In particular, $\cdr{S}$ is admissible with partition $X(\cdr{S})=X(S)$ and $Z(\cdr{S})=Z(S)\cup A(S)\cup B(S)$.
We will need the following lemma.
\begin{Lemma}\label{Lemmautile}
Suppose we have $\alpha,\beta,\gamma\in\Psi$ such that $\suc(\alpha)>\suc(\beta)>\suc(\gamma)$ and $\alpha\ngtr\beta,\beta\ngtr\gamma$.
Then $\alpha\ngtr\gamma$.
\end{Lemma}
\begin{proof}
Note that it must be $\alpha,\beta$ short.
Then we can write $\alpha=e_i+e_j,\suc(\alpha)=2e_i,\beta=e_h+e_k,\suc(\beta)=2e_h$ and the hypothesis imply $i<h<k<j$.

Now suppose $\gamma=e_r+e_t$ short and $\suc(\gamma)=2e_r$ with $r<t$.
Then $\suc(\gamma)<\suc(\beta)$ implies $r>h$ and $\beta\ngtr\gamma$ implies $t<k$.
Hence, $i<h<r<t<k<j$.
If $\gamma=\suc(\gamma)=2e_r$ a similar reasoning gives $i<h<r<k<j$.
In both cases the claim is proved.
\end{proof}

\begin{theorem}\label{completion}
Suppose that $S=X(S)\sqcup Z(S)$ is admissible for $v$.
Then $\cdr{S}$ is full admissible for $v$.
\end{theorem}
\begin{proof}
It is clear by the definition that $v(\cdr{S})<0$, so we want to show $A(\cdr{S})=B(\cdr{S})=\varnothing$.

Suppose that we have $\gamma\in X(\cdr{S})=X(S)$ long and $\alpha\in \cdr{S}$ short with $\suc(\alpha)\in\Phi^+(v)$, $\suc(\alpha)>\gamma$, $\alpha\nless\gamma$ and $\suc(\alpha)\notin Z(\cdr{S})$.
Then $\alpha\in S$ because to obtain $\cdr{S}$ we only added long roots.
Hence, $\suc(\alpha)\in A(S)\subseteq Z(\cdr{S})$ and we have a contradiction.

On the other hand, suppose we have $\gamma\in Z(\cdr{S})=Z(S)\cup A(S)\cup B(S)$ long and $\alpha\in \cdr{S}$ short with $\suc(\alpha)\in\Phi^+(v)$, $\suc(\alpha)>\gamma$, $\pad(\gamma)\nless \alpha$ and $\suc(\alpha)\notin Z(\cdr{S})$.
As before, $\alpha\in S$, hence we can have $\suc(\alpha)\notin Z(\cdr{S})$ only if $\gamma\notin Z(S)$.
But Lemma \ref{Lemmautile} shows us that in this hypothesis $\gamma\in A(S)$ and $\gamma\in B(S)$ both imply $\suc(\alpha)\in A(S)\cup B(S)\subseteq Z(\cdr{S})$ and that's a contradiction.
\end{proof}

We have an easy corollary.
\begin{Corollary}\label{compcomp}
Let $S$ be admissible.
Then $\cdr{S}=S$ if and only if $S$ is full admissible.
In particular, $\cdr{\cdr{S}}=\cdr{S}$ for every $S$ admissible.
\end{Corollary}
\begin{proof}
We already saw that $A(S)\cup B(S)=\varnothing$ if and only if $S$ is full admissible.
Then, the fact that $\cdr{\cdr{S}}=\cdr{S}$ is a consequence of Theorem \ref{completion}.
\end{proof}

Now, return to our pairs of examples $a1)$, $a2)$ and $b1)$,$b2)$.
In both examples we are adding a single root denoted with $Z_1$.
However, the completion from $a2)$ to $a1)$ gives $A=\left\lbrace Z_1\right\rbrace$ and $B=\varnothing$, while the completion from $b2)$ to $b1)$ gives $A=\varnothing$ and $B=\left\lbrace Z_1\right\rbrace$.
This two examples completely model what may happen in general.
Specifically, the way we showed that the pairs of elements $a1)$,$a2)$ and $b1)$,$b2)$ are in the same $B_v$-orbit can be easily extended to generic elements of a $B_v$-orbit in $\mathfrak{p}^u$ to obtain the following result.

\begin{Lemma}
Let $S$ be admissible for $v$ and $\cdr{S}$ its full admissible completion.
Then $B_ve_S=B_ve_{\cdr{S}}$.
\end{Lemma}

Now, if $S=X\sqcup Z$ is admissible we will say that a root $\beta\in Z$ is \textit{essential} if $\beta\notin\cdr{S\setminus\left\lbrace\beta\right\rbrace}$.
This is equivalent by Corollary \ref{compcomp} to asking that $\cdr{S\setminus\left\lbrace\beta\right\rbrace}\neq \cdr{S}$.
From what we said above we see that if $\beta$ is not essential then $S\setminus\left\lbrace\beta\right\rbrace$ is still admissible and the elements $e_{S}$ and $e_{S\setminus\left\lbrace\beta\right\rbrace}$ are in the same orbit. 

Moreover, note that if $\beta$ is essential in $S$ and $\beta<\alpha\in\Psi$, then $\beta$ is essential also in $S\cap \left\lbrace \gamma\in\Psi\mid \gamma<\alpha\right\rbrace$.


Now suppose $S$ admissible.
The following easy lemma will be useful later.

\begin{Lemma}\label{independentC}
Let $S$ be admissible for $v$.
Then $S$ is linearly independent in $\Phi\otimes\mathbb{R}$.
\end{Lemma}
\begin{proof}
For every $\beta_i\in Z(S)$ there is $\gamma_i\in\Phi^+_P$ such that $\pad(\beta_i)+\gamma_i=\beta_i$.
Then $S$ is linearly independent if and only if $S'=\left(S\setminus \bigcup\left\lbrace\beta_i\right\rbrace\right)\cup\bigcup\left\lbrace\gamma_i\right\rbrace$ is linearly independent.
But this is clear because the elements of $S'$ are all pairwise orthogonal.
%
\end{proof}

The following lemma is the key to prove that the full admissible pairs parametrize the orbits.
It basically gives us an algorithm to obtain an admissible representative of an orbit.
Then, we saw above that we can complete it to a full admissible representative for the same orbit.
Note that if $R\subseteq \Phi$ and $\alpha\in \Phi$ we will write $R>\alpha$ meaning that $\beta>\alpha$ for every $\beta\in R$.

\begin{Lemma}\label{corBL}
Suppose $\beta\in \Phi^+(v)$ and $S$ admissible such that for every $\alpha\in S$, $\beta\nless \alpha$.
Moreover, suppose that $S\cup\left\lbrace\beta\right\rbrace$ is not admissible . 
Then, for every $t\in\mathbb{K}$ there is $u\in B_L$ such that
$$\supp(u\left(e_S+te_\beta\right)-e_S)>\beta$$
\end{Lemma}
\begin{proof}
Note that if $\gamma\in\supp(x)$ is minimal, then $\gamma\in \supp(ux)$ for every $u\in B_L$.
It follows that the thesis is equivalent to proving the existence of $u\in B_L$ such that
$$\supp(ue_S-e_S-te_\beta)>\beta$$
to see this, simply apply $u^{-1}\in B_L$.

Now put $X=X(S)$, $Z=Z(S)$ and $S'=S\cup \left\lbrace\beta\right\rbrace$ .
It must be $X\cup \left\lbrace\beta\right\rbrace$ not orthogonal or $S'$ would be admissible with partitions $X(S')=X\cup \left\lbrace\beta\right\rbrace$ and $Z(S')=Z$.
Similarly $\beta$ must be short or $X(S')=X$ and $Z(S')=Z\cup\left\lbrace\beta\right\rbrace$ would be an admissible partition for $S'$.
Then there exists $\gamma\in\Phi^+_P$ and $\alpha\in X$ such that $\beta=\alpha+\gamma$.
Note that $\alpha'+\gamma\notin \Phi$ for every $\alpha'\in X\setminus\left\lbrace \alpha\right\rbrace$.

If $\alpha\notin \pad(Z)$, then $u_\gamma(s).e_S=e_S+ase_{\beta}+bs^2e_{\beta+\gamma}$ for some constant $a\in\mathbb{K}^*$ and $b\in \mathbb{K}$.
It follows that there is $s_0\in\mathbb{K}^*$ for which
$$\supp(u_\gamma(s_0).e_S-\left(e_S+te_\beta\right))>\beta$$

If there is $\delta\in Z$ with $\alpha=\pad(\delta)$, we have $u_\gamma(s).e_\delta=e_\delta+cse_{\delta+\gamma}+ds^2e_{\delta+2\gamma}$ for some $c,d\in\mathbb{K}$ (not necessarily different from zero).
But $\beta=\alpha+\gamma<\delta+\gamma<\delta+2\gamma$, so we can still find $s_0\in\mathbb{K}^*$ such that $u_\gamma(s_0).e_{S}$ has the property we are looking for.
\end{proof}

\begin{theorem}\label{primaparte}
Every orbit $\mathcal{O}$ contains an element of the form $e_S$ where $S$ is admissible.
\end{theorem}
\begin{proof}
Take an element $x\in\mathcal{O}$. 
We can assume without loss of generality that $S=\supp(x)\subseteq\Phi^+(v)$. 
If $S$ is admissible, by Lemma \ref{independentC} $S$ is independent, hence there is an element $t$ in the torus $T$ such that $t.x=e_S$ and the claim is proved. 
Suppose $S$ not admissible and define an ascending chain in $S$

$$\left\lbrace \begin{array}{ll}
                  S_1=\min(S)\\
                  S_{i+1}=S_i\cup \min (S\setminus S_i)
                  
                \end{array}
              \right.
$$
Put $x=\sum_{\alpha\in S}a_\alpha e_\alpha$.

Given that the elements of $S_1$ are pairwise incomparable, $S_1$ must be orthogonal.
Hence, there must be an $i_0\geq 1$ such that $S_{i}$ is admissible for every $i\leq i_0$ but $S_{i_0+1}$ is not. 
Note that at most one long root can be in $\left(S_{i_0+1}\setminus S_{i_0}\right)$ because the long roots are always comparable.

Denote $\mathcal{S}_{i_0+1}=\left\lbrace \beta\in \left(S_{i_0+1}\setminus S_{i_0}\right)\mid \left(S_{i_0}\cup \left\lbrace \beta\right\rbrace\right) \text{ is admissible}\right\rbrace$. 
It is easy to see that $S_{i_0}\cup\mathcal{S}$ is also admissible.
For, there is at most a long root $\beta\in\mathcal{S}_{i_0+1}$ and all the other roots must be orthogonal to $S_{i_0}$.
Given that all the roots in $\mathcal{S}_{i_0+1}$ are pairwise orthogonal it is clear that $S'=S_{i_0}\cup\left\lbrace\mathcal{S}_{i_0+1}\setminus\left\lbrace\beta\right\rbrace\right)$ is admissible.
Now, $\beta$ is a long root which is not smaller than any root in $S'$, hence $S'\cup\left\lbrace\beta\right\rbrace$ is still admissible.
Given that $S'$ is admissible, we can suppose by Lemma \ref{independentC} that $a_\alpha=1$ for every $\alpha\in S'$.
 
Now if $\beta\in \left(S_{i_0+1}\setminus S'\right)$ we know by Lemma \ref{corBL} that there is $u\in B_L$ such that \[\supp\left(u.(e_{S'}+a_\beta e_\beta)-e_{S'}\right)>\beta\]
Then $u.x$ is such that $S_j(u.x)=S_j(x)$ for every $j\leq i_0$ and $S_{i_0+1}(u.x)=S_{i_0+1}(x)\setminus\left\lbrace\beta\right\rbrace$.
By induction we obtain the thesis.
\end{proof}

We have an easy corollary that comes directly from the proof of Theorem \ref{primaparte} and will be useful later.

\begin{Corollary}\label{corprimaparte}
Let $\mathcal{O}$ be a $B_v$-orbit and $x\in\mathcal{O}$.
Suppose $S\subseteq\supp (x)$ admissible such that for every $\alpha\in \left(\supp (x)\setminus S\right)$ and for every $\gamma\in S$ we have $\alpha\nleq\gamma$.
Then, if $\beta\in \min \left(\supp (x)\setminus S\right)$ and $S\cup \left\lbrace\beta\right\rbrace$ is admissible there is $T\supseteq S\cup \left\lbrace\beta\right\rbrace$ admissible such that $\mathcal{O}=B_ve_T$.
\end{Corollary}
\begin{proof}
Put $R=\supp(x)$.
Then, with the notation of Theorem \ref{primaparte} there is a minimum $j$ such that $S\subseteq R_j$.

If $S=R_J$, then $\beta\in\mathcal{S}_{j+1}$, hence the proof of Theorem \ref{primaparte} gives us an admissible set $T\supseteq S\cup\left\lbrace\beta\right\rbrace$ such that $\mathcal{O}=B_ve_T$.

If instead $S\subsetneq R_j$, given that there is no element in $R\setminus S$ smaller than a element in $S$ and $S$ is admissible it must be $R_{j-1}\subseteq S\subseteq R_{j-1}\cup\mathcal{S}_{j}$.
Moreover, given that $S\cup\left\lbrace \beta\right\rbrace$ is also admissible and $\beta$ is minimal it must be $S\cup\left\lbrace \beta\right\rbrace\subseteq R_{j-1}\cup\mathcal{S}_{j}$.
As before, the proof of Theorem \ref{primaparte} shows us that there is $T\supseteq R_{j-1}\cup\mathcal{S}_{j}$ such that $\mathcal{O}=B_ve_T$.
\end{proof}

We saw that there is a $P$-equivariant map
$$\exp\colon \mathfrak{p^u}\longrightarrow P^\mathfrak{u}$$
Recall that $B_v=B_L\ltimes U_v$, hence if $x\neq y\in \mathcal{O}$ we have $Bv\exp(x)v^{-1}B=Bv\exp(y)v^{-1}B$ and that if $S=\supp(x)$ is orthogonal, then $Bv\exp(x)v^{-1}B=Bv\sigma_{S}v^{-1}B$ where $\sigma_S=\prod_{\alpha\in S} s_\alpha$ is the involution related to $S$.

Consider $S=X\cup Z$ admissible and recall that for every element $\alpha\in S$ there is at most one other $\beta\in S$ such that $(\alpha,\beta)\neq 0$.
In this case, we can suppose $\alpha<\beta$ and we have $\alpha\in X$, $\beta\in Z$.
Denote $\gamma=\beta-\alpha$.
We have $s_{\alpha}(\alpha+\gamma)=\gamma-\alpha\in-\Phi^+(v)$ and $s_{\alpha+\gamma}(\alpha)=-\gamma\in \Phi^-_P$.
Note that this follows only from the relative length of the roots involved, so it is true even in type $\bf{B}$.

\begin{Lemma}\label{involuzione}
Fix an orbit $\mathcal{O}=B_ve_S$ where $S=X\cup Z$ is admissible for $v$.
Then,
$$Bv\exp(e_S)v^{-1}B=Bv\sigma_{X}v^{-1}B=B\sigma_{v\left(X\right)}B$$
\end{Lemma}
\begin{proof}
We have $Bv\exp(e_S)v^{-1}B=Bv\exp(e_Z)\exp(e_X)v^{-1}B$.
By hypothesis, $v(S)<0$ so both $v(-X)$ and $v(-Z)$ are in $B$.
Recall that we can chose $x_\alpha$ such that 
$$\exp(-x_{-\alpha})\exp(x_{\alpha})\exp(-x_{-\alpha})=s_{\alpha}$$
where $s_{\alpha}$ is a representative in the normalizer $N_{G}(T)$ for the reflection relative to $\alpha$.
The sets $X$ and $Z$ are orthogonal (relatively to themselves), so we can write:
$$Bv\exp(e_Z)\exp(e_X)v^{-1}B=Bvs_Z\exp(e_{-Z})\exp(e_{-X})s_{X}v^{-1}B$$
Note that the terms $\exp(e_{-Z})$ and $\exp(e_{-X})$ commute so
\begin{align*}
Bvs_Z\exp(e_{-Z})\exp(e_{-X})s_{X}v^{-1}B &=Bvs_Z\exp(e_{-X})\exp(e_{-Z})s_{X}v^{-1}B\\
&=Bv\exp(e_{-s_{Z}(X)})s_Zs_{X}\exp(e_{-s_X(Z)})v^{-1}B
\end{align*}
Note now that $-vs_Z(X)>0$ because an element of $X$ is either fixed by $s_Z$ or mapped in $\Phi^-_P$ by the discussion above and $-vs_X(Z)<0$ because every element in $Z$ is mapped in $-\Phi^+(v)$.
We have
\begin{align*}
Bv\exp(e_{-s_{Z}(X)})s_Zs_{X}\exp(e_{-s_X(Z)})v^{-1}B &=Bvs_Zs_{X}\exp(e_{s_X(Z)})s_{s_X(Z)}v^{-1}B\\
&= Bvs_Z\exp(e_{Z})s_{X}s_{s_X(Z)}v^{-1}B \\
&= Bv\exp(e_{-Z})s_Zs_Xs_{s_X(Z)}v^{-1}B=\\
&=Bvs_Zs_Xs_{s_X(Z)}v^{-1}B
\end{align*}
because $-v(Z)>0$.
The claim follows from $s_Zs_Xs_{s_X(Z)}=s_Zs_X(s_Xs_{Z}s_X)=s_X$.
\end{proof}

It follows from Corollary \ref{SugualeT} that if $B_ve_S=B_ve_T$ with $S$ and $T$ admissible then $X(S)=X(T)$. 
To prove the last part of the classification we will need the following technical lemma.

\begin{Lemma}\label{tecnico}
Fix $v\in W^P$ and let $\left(S,2e_i,e_i+e_j,b\right)$ be a quadruple with $S\subseteq \Phi^+(v)$ admissible, $e_i+e_j,2e_i\in\Phi^+(v)\setminus S$ and $b\in B_v$ such that:
\begin{enumerate}
\item $S<2e_i$;
\item $S\cup \left\lbrace 2e_i,e_i+e_j\right\rbrace$ is admissible and $2e_i$ is essential;
\item $S\cup\left\lbrace 2e_i\right\rbrace\subseteq\supp(b.e_S)$.
\end{enumerate}
Then there is $\alpha\in\supp(b.e_S)\setminus S$ such that $\alpha<2e_i$ and $\alpha\ngeq e_i+e_j$.
\end{Lemma}
\begin{proof}
Write $\beta=2e_i$, $\pad(\beta)=e_i+e_j$ and $T=S\cup\left\lbrace e_i+e_j\right\rbrace$.

Consider the set 
\[K=\left\lbrace b\in B_v\mid S\cup\left\lbrace 2e_i\right\rbrace\subseteq\supp(b.e_S)\right\rbrace\]
which by assumption is non-empty.
Given that $T\cup \left\lbrace \beta\right\rbrace$ is admissible, $\beta$ must be orthogonal to $S$, so $S\cup\left\lbrace\beta\right\rbrace$ must be admissible and $X(S)\neq X(S\cup\left\lbrace\beta\right\rbrace)$.
Hence, by Corollary \ref{corprimaparte} and Lemma \ref{involuzione} for every $b\in K$ there must be a minimal root $\alpha_b\in \supp(be_S)$ such that $S\bigcup \left\lbrace \alpha_b \right\rbrace$ is not admissible and $\alpha_b<\beta$.

Then the claim is equivalent to saying that 
\[K'=\left\lbrace b\in K\mid \alpha_b\geq e_i+e_j\right\rbrace=\varnothing\]

Suppose by contradiction that $K'\neq\varnothing$ and fix $b_0\in K'$ such that $\alpha_{b_0}$ is maximal among the $\alpha_b$ for $b\in K'$.
Finally, denote for simplicity $\alpha=\alpha_{b_0}$.

Note that $\alpha\neq e_i+e_j$ because $S\cup\left\lbrace e_i+e_j\right\rbrace$ is admissible.
Then $\alpha=e_i+e_k$ with $j>k>i$.
The situation can be represented with the following diagram.
\[
\Yboxdimx{28pt}
\Yvcentermath1
\gyoung(;;{\pad(\beta)};\alpha;;\beta,;;;;,;;;,;;,;)
\]

where the elements in $S$ are not drawn, but we know they can't be on the upper row because they must be orthogonal to $\pad(\beta)$.
There must be $\tau\in X(S)$ that is not orthogonal to $\alpha$.
More precisely, there is $\gamma\in\Phi^+_P$ such that $\alpha=\tau+\gamma$.
We will sign with a $\star$ the possible positions of $\tau$
\[
\Yboxdimx{28pt}
\Yvcentermath1
\gyoung(;;{\pad(\beta)};\alpha;;\beta,;;;\star;,\star;\star;\star,;;,;)
\]

We claim that the one-parameter subgroup $u_\gamma(t)$ acts as the identity on $e_{S\setminus \left\lbrace\tau\right\rbrace}$.
This is clear if $\suc(\tau)\notin Z(S)$, so suppose $\suc(\tau)\in Z(S)$.
If we write $\suc(\tau)=2e_r$ and $\tau=e_r+e_s$ with $r<s$, then we know that $r>i$.
Moreover, given that $2e_i$ is essential, it must be $\tau<e_i+e_j$, hence $j<s$.
But now $\gamma=e_r+e_s$ and $\alpha=e_i+e_k$ are not orthogonal by hypothesis, and the only way to obtain this is if $r=k$ (because $s>j>k>i$).
Then $\gamma=e_i-e_s$ and is again the identity on $e_{2e_k}$.
Schematically, this means that if $\suc(\tau)\in Z(S)$ then we cannot be in the situation of Diagram $a)$, hence we must be in the situation of Diagram $b)$
\[\begin{array}{rlccrl}
a) & 
\Yboxdimx{28pt}
\Yboxdimy{16pt}
\Yvcentermath1
\gyoung(;;{\pad(\beta)};\alpha;;\beta,;;;\tau;{\suc(\tau)},;;;,;;,;) & & & b) & \Yboxdimx{28pt}
\Yboxdimy{16pt}\Yvcentermath1\gyoung(;;{\pad(\beta)};\alpha;;\beta,;;;;,\tau;;{\suc(\tau)},;;,;)
\end{array}\]
because $\beta$ is not essential in Diagram $a)$.

Note that given how we defined $\alpha$, every root in $\supp(b_0.e_S)$ that is smaller than $\beta$ is either greater than $\alpha$, $\alpha$ itself or in $S$, so $u_\gamma(t)$ must act as the identity on every such root different from $\tau$.
Hence, there is $t\in\mathbb{K}^*$ such that $2e_i\in\supp(u_\gamma(t)b.e_S)$ and  $S\subseteq \supp(u_\gamma(t)b_0.e_S)$, so $b'=u_\gamma(t)b_0\in K$.
But $\alpha_{b'}>\alpha$ and that's absurd because $\alpha=\alpha_{b_0}$ was maximal.
\end{proof}

\begin{theorem}\label{secondaparte}
Let $S,T\subseteq\Phi^+(v)$ be full admissible. 
Then 
$$B_ve_S=B_ve_T \text{ if and only if } S=T$$
\end{theorem}
\begin{proof}
Suppose by contradiction that $T\neq S$ and take $\beta$ minimal in $\left(T\setminus S\right)\cup \left(S\setminus T\right)$.
Without loss of generality, we can suppose $\beta\in T\setminus S$.

Consider $b\in B_v$ such that $be_S=e_T$.
By Lemma \ref{involuzione} and Corollary \ref{SugualeT} it must be $X(S)=X(T)$, hence actually $\beta\in Z(T)\setminus Z(S)$.
Then if $M_\beta=\left\lbrace \alpha\in \Psi\mid \alpha<\beta\right\rbrace$ it must be $S\cap M_\beta=T\cap M_\beta\doteqdot M$.
Note that both $M$ and $M\cup\left\lbrace\beta\right\rbrace$ are admissible.

Now, $\beta$ is long, so it must be of the form $\beta=2e_i$ and if $\pad(\beta)=e_i+e_j\in X(T)$, then $e_i+e_j\in X(S)$ also.
Denote $M'=M\setminus\left\lbrace e_i+e_j\right\rbrace$.
We want to apply Lemma \ref{tecnico} to the quadruple $\left(M', \beta,\pad(\beta), b\right)$ to obtain a contradiction.
We know by hypothesis that $2e_i\in \supp(be_M)$ and $2e_i\notin \supp(be_{e_i+e_j})$ by Lemma \ref{cortelunghe}.
Hence, $2e_i\in \supp(be_{M'})$.
Moreover, $\beta$ must be essential for $M\cup\left\lbrace\beta\right\rbrace$ because if $\beta\in\cdr{M}$, then also $\beta\in\cdr{S}=S$.

It follows by Lemma \ref{tecnico} that there is $\alpha\in \supp(b e_{M'})\setminus M'$ such that $\alpha<\beta$ and $\alpha\ngeq e_i+e_j$.
But $\alpha\ngeq e_i+e_j$ implies $\alpha\notin\supp(be_{e_i+e_j})$, hence $\alpha\in\supp(be_M)\setminus M$.
On the other hand, it is clear that if $\alpha\in\supp(be_\gamma)$, then $\gamma<\alpha<\beta$, hence $\alpha$ can't be in the support of $be_\gamma$ for every $\gamma\nless\beta$.
This means that we must have $\alpha\in \supp(be_S)= T$ and that's a contradiction.
\end{proof}



\section{The dimension of the $B_v$-orbits}

We want to calculate the dimension of a $B_v$-orbit in $\mathfrak{p}_u$ or equivalently $B$-orbits in $G/L$.

%

In the next couple of paragraphs, we drop the hypothesis that the field $\mathbb{K}$ has characteristic 2.
\begin{Definition}\label{good}
For every $v\in W^P$ let $\mathcal{S}^v=\left\lbrace S^v_i\right\rbrace_i$ be a family of subsets $S^v_i\subseteq \Phi^+(v)$ that parametrizes the $B_v$-orbits in $\mathfrak{p}^u$ through the map $S^v_i\mapsto B_ve_{S^v_i}$.
Let $\mathcal{S}$ be the family of these parametrizations.
We know that $\mathcal{S}$ parametrizes the $B$-orbits in $G/L$.

We will say that $\mathcal{S}$ is a \emph{good} parametrization if:
\begin{enumerate}
\item for every $u<v$ the map
\begin{align*}
S^v&\longrightarrow S^u\\
S_i^v&\longmapsto S_i^v\cap \Phi^+(u)
\end{align*}
is well-defined;
\item  $\beta$  maximal in $\Phi^+(v)$ implies $S^v\cup\left\lbrace\beta\right\rbrace$ independent for every $S^v\in\mathcal{S}^v$;
\item \label{good2} $B_ve_{S^v\cup \left\lbrace \beta\right\rbrace}=B_ve_{S^v}$ implies that there is $b\in B_L$ such that 
$$\supp(be_{S^v}-e_{{S^v}\cup \left\lbrace \beta\right\rbrace})>\beta$$
\end{enumerate}
\end{Definition}

\begin{Lemma}
The following are good parametrizations:
 \begin{enumerate}
 \item the parametrization for the type $B$-case given in \ref{parametrizzazioneB};
 \item the parametrization for the type $C$-case given in \ref{parametrizzazioneC};
 \item the parametrization for the non-characteristic-$2$ cases given in \cite{GM}
\end{enumerate}
\end{Lemma}
We can now prove the dimension formula.

\begin{theorem}\label{theodim}
Let $\mathcal{S}$ be a good parametrization and, with the notation above, fix $v\in W^P$ and $S\in\mathcal{S}^v$.
Then, the dimension of $Bvx_S$ is
$$\dim Bvx_S=\#\Psi+\#Y(v,S)$$
where $Y(v,S)=\left\lbrace \beta\in \Phi^+(v)\mid \exists b\in B_L \text{ such that } \supp(be_{S\cup \left\lbrace \beta\right\rbrace}-e_S)>\beta \right\rbrace$.
\end{theorem}
\begin{proof}
We will show the claim by induction on $l(v)$.
If $l(v)=0$, then $Y=\varnothing$ and $Bvx_S=BL/L\cong B/B_L\cong\mathfrak{p}_u$ so the formula holds.

Now suppose $l(v)>0$ and $\alpha\in\Delta$ such that $s_\alpha v<v$.
Denote $\beta=v^{-1}(-\alpha)$.
Then $P_\alpha=Bs_\alpha\sqcup BU_{-\alpha}$, so
\begin{align*}
P_\alpha vx_S&=Bs_\alpha vx_{S\setminus\left\lbrace \beta\right\rbrace} \cup BU_{-\alpha}vx_S\\
&=Bs_\alpha vx_{S\setminus\left\lbrace \beta\right\rbrace} \cup \bigcup_{t\in\mathbb{K}}Bvu_{\beta}(t)x_S\\
&=Bs_\alpha vx_{S\setminus\left\lbrace \beta\right\rbrace}\sqcup Bvx_S\cup\bigcup_{t\in\mathbb{K}^*}Bvu_{\beta}(t)x_S\\
&=Bs_\alpha vx_{S\setminus\left\lbrace \beta\right\rbrace}\sqcup Bvx_{S\cup\left\lbrace \beta\right\rbrace}\sqcup Bvx_{S\setminus\left\lbrace \beta\right\rbrace}
\end{align*}
where we used the fact that $\beta\notin \Phi^+(s_\alpha v)$ and that $S\cup\left\lbrace\beta\right\rbrace$ is independent.

The $B$-orbit $Bs_\alpha vx_{S\setminus\left\lbrace \beta\right\rbrace}$ can't be the open one because if $\cdr{Bs_\alpha vx_{S\setminus\left\lbrace \beta\right\rbrace} }\supseteq Bvx_S$ then it follows $\cdr{Bs_\alpha vP}\supseteq BvP$ and that's impossible because $s_\alpha v<v$.

We show at first that $\beta\in Y(v,S)$ if and only if $Bvx_S=Bvx_{S\cup\left\lbrace\beta\right\rbrace}$.
It is clear that if we take the element $b\in B_L$ for which $\supp(be_{S\cup \left\lbrace \beta\right\rbrace}-e_S)>\beta$, then $\supp(be_{S\cup \left\lbrace \beta\right\rbrace}-e_S)\in\Psi\setminus \Phi^+(v)$ which means that $B_ve_S=B_ve_{S\cup\left\lbrace\beta\right\rbrace}$ and that's the same as $Bvx_S=Bvx_{S\cup\left\lbrace\beta\right\rbrace}$.
On the other hand by \ref{good2} in definition \ref{good} we have $Bvx_S=Bvx_{S\cup\left\lbrace\beta\right\rbrace}\Rightarrow B_ve_S=B_ve_{S\cup\left\lbrace\beta\right\rbrace}\Rightarrow \beta\in Y(v,S)$.

It follows that, if $\beta\in Y(v,S)$ then $Bvx_S$ is the open orbit and
$$\dim Bs_\alpha vx_{S\setminus \left\lbrace\beta\right\rbrace}=\dim Bvx_S-1$$

By inductive hypothesis $\dim Bs_\alpha vx_{S\setminus \left\lbrace\beta\right\rbrace}=\#\Psi+\#Y(s_\alpha v, S)$.
We conclude by noting that $Y(s_\alpha v, S)=Y(v, S)\setminus \left\lbrace\beta\right\rbrace$.

If instead $\beta\notin Y(v,S)$, then $Bv\exp(e_S+te_\beta)\neq Bvx_S$ and it is the open orbit.
It follows 
$$\dim Bs_\alpha vx_{S\setminus \left\lbrace\beta\right\rbrace}=\dim Bvx_S$$
as before $\dim Bs_\alpha vx_{S\setminus \left\lbrace\beta\right\rbrace}=\#\Psi+\#Y(s_\alpha v, S)$, but now $Y(s_\alpha v, S)=Y(v,S)$ and we are done.
\end{proof}

\begin{Corollary}
$$\dim B_ve_S=\#\left(\Psi\setminus \Phi^+(v)\right)+\#Y(v,S)$$
\end{Corollary}
\begin{proof}
It is sufficient to note that $\#\Phi^+(v)=l(v)$.
\end{proof}

Now consider the (good) parametrization of \cite{GM} in admissible pairs.
We have two different ways of computing the dimension, that is
$$\dim(Bvx_S)=\#\Psi+L(\sigma_{v(S)})$$
and
$$\dim(Bvx_S)=\#\Psi+\#Y(v,S)$$
This means that if the characteristic is different from $2$ it must be
$$\#Y(v,S)=L(\sigma_{v(S)})$$
It follows that studying how $Y(v,S)$ varies for similar orbits when the characteristic is (or isn't) $2$ should give us a more explicit formula.

We will study at first the case where the characteristic of the base field $\mathbb{K}$ is not $2$.
Then we will be able to give another description of $Y(v,S)$ that is more combinatoric

\begin{Lemma}\label{Ycharinf}
If $\ch \mathbb{K}\neq 2$ we have
\[Y(v,S)=\left\lbrace\beta\in\Phi^+(v)\mid \beta\in S \text{ or there is }\alpha\in S\text{ such that }\beta-\alpha\in\Phi^+\right\rbrace\]
\end{Lemma}
\begin{proof}
The containment $\supseteq$ is clear whenever $\beta\in S$, while if there is $\alpha\in S$ such that $\beta-\alpha\in\Phi^+$, then $\beta-\alpha\in\Phi^+_P$, so $u_{\beta-\alpha}(t)\subseteq B_L$ and there is $t\in\mathbb{K}$ such that $\supp(u_{\beta-\alpha}(t).e_{S\cup\left\lbrace\beta\right\rbrace}-e_S)>\beta$.
Note that $u_{\beta-\alpha}(t)$ fixes all $e_\gamma$ with $\gamma\in S,\gamma\neq\alpha$ because of orthogonality.

For the converse, suppose $\beta\in Y(v,S)$ but $\beta\notin S$ and for every $\alpha\in S$ either $\beta-\alpha$ is not a root or it is a negative root.
Then denote with $v_\beta\in W^P$ the smallest element for which $\beta$ is maximal, or, equivalently, the only element for which $\beta$ is the maximum of $\Phi^+(v_\beta)$ and put $S_\beta=S\cap \Phi^+(v_\beta)$.
By the fact that $\beta\in Y(v,S)$ we have $b\in B_L$ such that $\supp(be_{S\cup\left\lbrace\beta\right\rbrace}-e_S)>\beta$, but $B_L\subseteq B_{v_\beta}$ so $B_{v_\beta}e_{S_\beta}=B_{v_\beta}e_{S_\beta\cup\left\lbrace\beta\right\rbrace}$ and that's a contradiction because $\beta$ is orthogonal to all $\alpha\in S_\beta$.
\end{proof}

We saw that when the root system is of type $\bf{B}$, the set of orbit is small enough to make viable a case by case analysis.
Concordantly, it is possible to have a similar analysis for the dimension of the orbits.

\begin{Corollary}[Dimension formula for type $\bf{B}$]
If $\ch(\mathbb{K})=2$ and the root system is of type $B$ then
\[\dim(Bvx_S)=\#\Psi+H(v,S)\]
where 

\[H(v,S)=\left\lbrace \begin{array}{ll}
						L(\sigma_{v(S)}) & \text{ if }S \text{ is orthogonal and }S\neq \left\lbrace\alpha_0\right\rbrace\\
						L(\sigma_{v(s_{\alpha_0})})-\#\left\lbrace \alpha_0<\alpha<\gamma\right\rbrace &\text{ if }S=\left\lbrace\alpha_0,\gamma\right\rbrace\\
						L(\sigma_{v(s_{\alpha_0})})-\#\left\lbrace \alpha\in\Phi^+(v)\mid\alpha_0<\alpha\right\rbrace &\text{ if }S=\left\lbrace\alpha_0\right\rbrace
						\end{array}
						\right.
						\]
\end{Corollary}

We also have a similar result for the orbits in type $\bf{C}$.

\begin{theorem}[Dimension formula for type $\bf{C}$]\label{theodimC}
Fix $v\in W^P$ and $S$ full admissible for $v$.
Then:
\[\dim(Bvx_S)=\#\Psi+L(\sigma_{v(X(S))})-\#S_s+\#Z(S)\]
\end{theorem}
\begin{proof}
We can then apply theorem \ref{theodim} to get
\[\dim(Bvx_S)=\#\Psi+\#Y(v,S)\]
We want to give another description on $Y(v,S)$.
We claim that
\[Y(v,S)=\left\lbrace\beta\in\Phi^+(v) \mid \beta\in S\text{ or } \beta \text{ is short and there is }\alpha\in X(S) \text{ such that }\beta-\alpha\in\Phi^+\right\rbrace\]
The containment $\supseteq$ is clear if $\beta\in S$, so suppose $\beta$ short and $\alpha\in X(S)$ such that $\gamma=\beta-\alpha$ is a positive root.
Note that $u_{\beta-\alpha}(t)$ fixes all $e_\gamma$ with $\gamma\in X(S),\gamma\neq\alpha$ because of orthogonality and if $\suc(\alpha)\in S$ than $\supp(u_{\beta-\alpha}(t).e_{\suc(\alpha)}-e_{\suc(\alpha)})>\beta$.
Then there is $t\in\mathbb{K}$ such that $\supp(u_{\beta-\alpha}(t).e_{S\cup\left\lbrace\beta\right\rbrace}-e_S)>\beta$.

On the other hand suppose $\beta \in Y(v,S)$, and suppose at first by contradiction that $\beta\notin S, \beta$ long.
Then, $S$ being full admissible $S\cup\left\lbrace \beta\right\rbrace$ is admissible and its full admissible completion represent a different orbit, so $\beta\notin Y(v,S)$.

Suppose instead $\beta$ short and for every $\alpha\in X(S)$ either $\beta-\alpha$ is not a root or it is a negative root.
Then for every $\alpha\in X(S)$ either $(\alpha,\beta)=0$ or $\beta<\alpha$.
So, consider $v_\beta\in W^P$ the element such that the unique maximal root in $\Phi^+(v_\beta)$ is $\beta$.
By the fact that $\beta\in Y(v,S)$ we have $b\in B_L$ such that $\supp(be_{S\cup\left\lbrace\beta\right\rbrace}-e_S)>\beta$, but $b\in B_{v_\beta}$ so $B_{v_\beta}e_S=B_{v_\beta}e_{S\cup\left\lbrace\beta\right\rbrace}$ and that's a contradiction because $\beta$ is orthogonal to all $\alpha\in X(S)\cap\Phi^+(v_\beta)$.

We now recall that, by lemma \ref{Ycharinf},
\[L(\sigma_{v(X(S))})=\#\left\lbrace\beta\in\Phi^+(v)\mid \beta\in S \text{ or there is }\alpha\in S\text{ such that }\beta-\alpha\in\Phi^+\right\rbrace\]
It follows that in this case $\#Y(v,S)$ is exactly $L(\sigma_{v(X(S))})$ minus the number of long roots $\beta\notin S$ such that there is $\alpha\in S$ and $\beta-\alpha\in\Phi^+$.
This is exactly the number of short roots $\alpha$ in $S$ for which $\suc(\alpha)\notin S$ or alternatively, the number of all short roots in $S$ minus the number of roots in $Z(S)$.
\end{proof}

\bibliography{bibliografia}{}
\bibliographystyle{unsrt}

\end{document}